\documentclass[a4paper, 12pt]{article}

\usepackage{amsfonts}
\usepackage {amssymb}
\usepackage {amsmath}
\usepackage {amsthm}
\usepackage{graphicx}
\usepackage{multirow}
\usepackage{xypic}
\usepackage {amscd}
\usepackage{mathrsfs}
\usepackage[colorlinks, linkcolor=blue, anchorcolor=black, citecolor=red]{hyperref}
\usepackage{enumerate}
\usepackage{enumitem}
\usepackage{geometry}  
\usepackage{tikz-cd}
\usepackage[all]{xy}

\usepackage{titlesec}

\newcommand{\bC}{\mathbb{C}}

\newcommand{\bZ}{\mathbb{Z}}
\newcommand{\bR}{\mathbb{R}}

\newcommand{\cO}{\mathcal{O}}
\newcommand{\cK}{\mathcal{K}}

\newcommand{\cA}{\mathcal{A}}

\newcommand{\Vol}{\mathrm{Vol}}
\newcommand{\Ric}{\mathrm{Ric}}
\newcommand{\loc}{\mathrm{loc}}
\newcommand{\supp}{\mathrm{supp}}
\newcommand{\trdeg}{\mathrm{trdeg}}
\newcommand{\idd}{i\partial\bar\partial}

\newcommand{\gC}{g_{\mathbb{C}}}

\newtheorem{thm}{Theorem}[section]
\newtheorem{prop}[thm]{Proposition}
\newtheorem{defn}[thm]{Definition}
\newtheorem{dfnthm}[thm]{Definition and Theorem}

\newtheorem{rem}[thm]{Remark}

\newtheorem{ques}[thm]{Question}
\newtheorem{exmp}[thm]{Example}
\newtheorem{lem}[thm]{Lemma}

\newtheorem{defn-prop}[thm]{Definition-Proposition}
\newtheorem{con}{Construction}[section]

\titleformat{\part}
  {\Large\bfseries}
  {\partname\ \thepart}
  {1em}
  {}

\titlespacing*{\part}{0pt}{3ex}{2ex}

\begin{document}

\title{Revised Demailly's Affineness Criterion and Algebraization of Entire Grauert Tubes}
\author{Kyobeom Song}
\date{}

\maketitle

\begin{abstract}
We provide a partial answer to Burns' 1982 conjecture on the affineness of entire Grauert tubes: the complement of a codimension-one subset of an entire Grauert tube is affine. This result is obtained by establishing a generalized version of Demailly's criterion for affineness of Stein manifolds, which may be of independent interest.

\end{abstract}
  
\section{Introduction}
A complex manifold $X$ is called a Stein manifold if it satisfies one of the following equivalent conditions:
\begin{itemize}
\item The holomorphic function ring $\mathcal{O}(X)$ yields a proper holomorphic embedding $X \hookrightarrow \mathbb{C}^N$ for some $N$.
\item There exists a strictly plurisubharmonic exhaustion function $\rho$ on $X$, i.e., $dd^c \rho > 0$ and $\{\rho \leq c\}$ is compact for all $c \geq 0$.
\item Cartan's Theorems A and B hold, i.e.,
\begin{itemize}
\item[(A)] For any coherent sheaf $\mathcal{F}$ and any open set $U \subset X$, the $\mathcal{O}_X$-module $\mathcal{F}(U)$ is generated by global sections.
\item[(B)] $H^i(X,\mathcal{F}) = 0$ for all $i > 0$.
\end{itemize}
\end{itemize}

One may notice that Stein manifolds in complex geometry correspond to affine varieties in algebraic geometry. Indeed, the analytification of any smooth affine variety is a Stein manifold. However, since Stein manifolds are defined in terms of arbitrary holomorphic functions rather than regular functions, they admit a much broader and more flexible class of examples than affine varieties. For instance, any pseudoconvex open subset of $\mathbb{C}^N$, any domain of holomorphy, as well as any closed complex submanifold of $\mathbb{C}^N$, is Stein.

Another example illustrating the flexibility of Stein manifolds is given by Grauert tubes. For any compact real-analytic Riemannian manifold $(M,g)$ of dimension $n$, there exists $\varepsilon > 0$ such that $T^\varepsilon M := \{ v \in TM \mid \|v\|_g < \varepsilon \}$ admits a Stein manifold structure called a \emph{Grauert tube}, satisfying the following properties \cite{LS91, S91, GS91, BHH03}:
\begin{enumerate}
\item $\rho := \|v\|_g^2$ is strictly plurisubharmonic and the K\"ahler metric associated to $\idd\rho$ restricts to $g$ on $M$,
\item $(i\partial\bar{\partial} \sqrt{\rho})^n = 0$ (the Homogeneous Complex Monge--Amp\`ere equation, HCMA).
\end{enumerate}
In particular, \emph{any} real-analytic manifold admits a Stein manifold that deformation retracts onto it, showing that the class of Stein manifolds is as vast as that of arbitrary real-analytic manifolds. If $\varepsilon = \infty$, i.e., if such a Stein structure exists on the entire tangent bundle $TM$, we call $T^\infty M := TM$ an \emph{entire Grauert tube}, and the triple $(TM, M, \rho)$ a Monge--Amp\`ere model.

In 1982, Burns \cite{B82} proposed the following conjecture.

\medskip

\noindent
\textbf{Burns Conjecture.} Is an entire Grauert tube an affine variety?

\medskip


This question is also mentioned in Demailly's 1985 paper, where necessary and sufficient conditions for a Stein manifold to be affine are established \cite[Probl\`eme 9.3]{D85}.

The interest of this problem lies both in the statement itself and in its broader connection to complex geometry. Intrinsically, a Stein manifold $X$ with maximal-dimensional real center $M:=\{\rho=0\}$ can be highly flexible, as seen in the case of Grauert tubes. In contrast, affine varieties with the same property are much more rigid. There is no a priori reason that the algebraic property of affineness should be characterized purely by an analytic condition such as the Monge--Amp\`ere equation. If true, this would reveal a nontrivial link between algebraic and analytic geometry, in the spirit of GAGA.

Moreover, this question is closely related to the Good Complexification Conjecture proposed by Totaro \cite{T03}. A closed smooth real manifold $M$ is said to admit a \emph{good complexification} if there exists an affine algebraic variety $U$ defined over $\mathbb{R}$ such that $U(\mathbb{R}) = M$ and the inclusion $U(\mathbb{R}) \hookrightarrow U(\mathbb{C})$ is a homotopy equivalence.

\medskip

\noindent
\textbf{Good Complexification Conjecture.} A closed manifold admits a metric of nonnegative curvature if and only if it admits a good complexification.

\medskip

This conjecture implies several known conjectures on nonnegatively curved manifolds. Kulkarni \cite{K78} showed that if a manifold $M$ admits a good complexification, then $\chi(M) \ge 0$, in accordance with the Hopf sign conjecture. Totaro \cite{T03} further proved that stronger restrictions are imposed on the rational cohomology ring. Thus, good complexifications can be viewed as an attempt to understand the rigidity of nonnegatively curved manifolds through underlying algebraic structures.

It is known that if $(M,g)$ admits an entire Grauert tube $TM$, then it should be nonnegatively curved \cite{LS91}, and $TM$ is homotopy equivalent to $M$. Hence, as observed by Totaro \cite{T03}, the Burns conjecture can be interpreted as asking whether an entire Grauert tube provides a good complexification. It is also remarkable that most known nonnegatively curved manifolds admit entire Grauert tubes, which also provide good complexification \cite{A01}.

Over the past 44 years, this problem has been approached from various perspectives. In 2000, Aguilar and Burns \cite{AB00} claimed a proof of the conjecture; however, Totaro \cite{T03} pointed out that Proposition 2.1, which is central to their argument, is not generally accepted. Later, in 2015 (published in 2018), Burns and Leung \cite{BL18} proved that if the base manifold $(M,g)$ is Zoll, i.e., all geodesics are periodic with the same period, then its entire Grauert tube is affine. In a series of works in the 2010s, Zelditch \cite{Z11, Z17} showed that Laplacian eigenfunctions on $(M,g)$ admit analytic continuations to the entire Grauert tube, and that the $\limsup$ of their absolute values with respect to $\sqrt{\rho}$ is exponentially bounded in $\sqrt{\rho}$. This suggests the abundance of functions on the entire Grauert tube with controlled growth, analogous to regular functions on affine varieties. In addition, it has recently been shown that entire Grauert tubes satisfy several strong topological restrictions (e.g., rational ellipticity \cite{C24, N22}, zero topological entropy of geodesic flow \cite{S24}).

Among these, Aguilar and Burns \cite{AB00} introduced a crucial nontrivial idea that motivates our approach. They attempted to prove the conjecture using the following affineness criterion of Demailly \cite{D85}.

\begin{thm}[Demailly's Affineness Criterion]\label{DAC}
Let $X$ be a Stein manifold of dimension $n$ (possibly with finitely many singularities). The following are equivalent:
\begin{enumerate}
\item $X$ is biholomorphic to an affine variety.
\item There exists a strictly plurisubharmonic exhaustion function $\phi$ on $X$ such that, writing $\alpha := dd^c \phi$ and $\beta := dd^c e^{\phi}$, the following conditions hold:
\begin{enumerate}
\item $X$ has finite Monge--Amp\`ere volume, i.e., $\int_X \alpha^n < \infty$,
\item there exists $\psi \in L^1_\loc(X,\mathbb{R}) \cap C^0(X_{\mathrm{reg}},\mathbb{R})$ and constants $A,B>0$ such that
\[
\psi \le A\phi + B\qquad \text{and}\qquad\mathrm{Ric}(\beta) + \frac{1}{2} dd^c \psi \ge 0,
\]
\item $\dim H^{2k}(X;\mathbb{R}) < \infty$ for all $k \in \mathbb{Z}$.
\end{enumerate}
\end{enumerate}
\end{thm}

Aguilar and Burns proposed to construct the function $\psi$ in condition (b) using the $(n,0)$-part of the analytic continuation of the volume form of the base manifold $(M,g)$ to the entire Grauert tube. However, for this approach to work, the resulting $(n,0)$-form $V^{n,0}$ must be globally well-defined. As we explain in Section \ref{egtdisc}, even when such a form can be defined, there exist explicit examples of entire Grauert tubes for which $V^{n,0}$ has poles along a singular set $S$. In this case, the function $\psi$ constructed from $V^{n,0}$ takes the value $-\infty$ on $S$, and hence Theorem \ref{DAC} cannot be applied.

To overcome this difficulty, we first establish a revised version of Demailly's affineness criterion (Theorem~\ref{DAC}), which is of independent interest.

\begin{thm}[Revised Affineness Criterion]\label{main}
Let $X$ be an $n$-dimensional Stein manifold equipped with a strictly plurisubharmonic exhaustion $\phi$.
Assume that $X$ has finite Monge--Amp\`ere volume
\[
\Vol(X)\;:=\;\int_X (dd^c\phi)^n \;<\;\infty.
\]
\begin{enumerate}
\item[(i)] Assume that there exist constants $A,B>0$ and a function $\psi\in L^1_\loc(X,\mathbb{R})$ such that, for $\beta:= dd^c(e^\phi)$, in the sense of $(1,1)$-currents, one has
\[
\Ric(\beta)+\frac{1}{2}dd^c\psi \;\ge\; 0 \qquad \text{and} \qquad \psi~\le~A\phi+B
\]
Then, writing $S:=\{x~|~e^{-2\psi}\notin L^1_\loc(x)\}$, for any $p\in X\setminus S$, there exists $b_p\in\mathcal{O}(X)$ with
$b_p(p)=1$ and $b_p|_S=0$ such that $X\setminus b_p^{-1}(0)$ is biholomorphic to a smooth affine variety that has $1/b_p$ as a regular function.

Furthermore, for any finitely many points $p_i\in X\setminus S$, there exists a normal affine variety $\cA$ such that each $X\setminus b_{p_i}^{-1}(0)$ embeds as a Zariski open subset of $\cA$.

\item[(ii)] If $\dim H^{2k}(X\setminus S;\mathbb{R})<\infty$ for all $k\in\mathbb{Z}$, then $X\setminus S$ is biholomorphic to a quasi-affine variety, i.e., a Zariski open set of an affine variety.

\item[(iii)] If $\dim H^{2k}(X\setminus S;\mathbb{R})<\infty$ for all $k\in\mathbb{Z}$ and $X\setminus S$ is a Stein manifold, then there exists an irreducible normal affine variety $\cA$ and its algebraic hypersurface $H$ such that $X\setminus S$ is biholomorphic to $\cA\setminus H$.

\item[(iv)] If $S=\emptyset$ and $\dim H^{2k}(X;\mathbb{R})<\infty$ for all $k\in\mathbb{Z}$, then $X$ is biholomorphic to an affine variety.
\end{enumerate}
\end{thm}

We then show that the metric $g$ on the base manifold $(M,g)$ extends meromorphically to a nondegenerate symmetric $(2,0)$-tensor $\gC$ on the entire Grauert tube $TM$. Using this, we prove that for the meromorphic nonvanishing section $\det(\gC)$ of $K_{TM}^2$, 
\[\phi:=\log(1+\cosh(\|v\|_g)),\qquad \beta:= dd^c (e^\phi),  \qquad\psi:=-\log\|\det(\gC)\|_\beta\]
satisfy the assumption of Theorem~\ref{main}(i), and the set $S:=\{x~|~e^{-2\psi}\notin L^1_\loc(x)\}$ is exactly the pole set of $\det(\gC)$. Hence, we obtain the following: an entire Grauert tube is affine outside a codimension-one subset.

\begin{thm}\label{egt}
Let $(M,g)$ be a compact Riemannian manifold admitting an entire Grauert tube $TM$. Then the metric $g$ extends to a nondegenerate meromorphic $(2,0)$-tensor $\gC$ on $TM$. Let $S$ denote the pole set of $\gC$. For any $p \in TM \setminus S$, there exists a holomorphic function $b_p \in \mathcal{O}(TM)$ such that $b_p(p)=1$ and $b_p|_S=0$, and the open set $TM \setminus b_p^{-1}(0)$ is biholomorphic to a smooth affine variety on which $1/b_p$ is a regular function.

Furthermore, for any finitely many points $p_i\in TM\setminus S$, there exists a normal affine variety $\cA$ such that each $TM\setminus b_{p_i}^{-1}(0)$ embeds as a Zariski open subset of $\cA$.
\end{thm}
For a detailed discussion of Theorem \ref{egt}, see Section \ref{egtdisc}.
\subsection{Discussion of Theorem \ref{main}}

Demailly's affineness criterion (Theorem~\ref{DAC}) provides a necessary and sufficient condition for a Stein manifold to be affine, and serves as a cornerstone connecting Stein and affine geometry. Among its hypotheses, the most essential is condition (b), a Ricci curvature bound. Conditions (a) finite volume, and (c) finite-dimensional cohomology are properties that are naturally expected for a Stein manifold presumed to be affine.

Condition (b) eliminates non-affine behavior in the Stein setting as follows. An affine variety admits a coordinate ring of regular functions and comes with a divisor at infinity, yielding strong control on both its function theory and boundary behavior. In contrast, for a Stein manifold, the analogue of the coordinate ring is $\mathcal{O}(X)$, which consists of arbitrary holomorphic functions, and the boundary can be highly pathological. 

The requirement that there exists a function $\psi$ with controlled growth $\psi \le A\phi + B$ such that $dd^c \psi$ bounds the Ricci form implies, via H\"ormander's $L^2$ estimates, that weighted $L^2$ holomorphic functions are dense in $\mathcal{O}(X)$. Moreover, control of the $L^2$ norms of holomorphic functions provides a way to regulate pathological behavior at the boundary of the Stein manifold.

However, condition (b) is also the most difficult to verify in practice. The usual strategies for constructing such a function $\psi$ are as follows:
\begin{enumerate}
\item Finding a global Ricci potential $\psi$ with controlled growth $\psi\le A\phi+B$ such that $dd^c \psi = -\Ric(\beta)$.
\item Alternatively, one may use multiple sections of the determinant line bundle $K_X^{-m}$. Recall that such bundles carry a natural norm induced by the Kähler form $\beta$. Since $\Ric(\beta)$ is locally expressed as $-\idd \log \det(\beta)$, it follows that for a section $s_i \in H^0(X,K_X^{-m})$ which does not vanish at a given point, the function
\[
\psi_i := \frac{1}{m}\log \|s_i\|_\beta^2
\]
satisfies $\idd \psi_i = -\Ric(\beta)$ away from the zero set of $s_i$. Therefore, by choosing finitely many sections $s_i$ with no common zero and such that $\|s_i\|_\beta \le e^{A\phi+B}$, one may define $\psi := \max_i \psi_i$.
\end{enumerate}

Both approaches are extremely difficult to realize on a general Stein manifold. The first condition would imply the vanishing of the real first Chern class, which is a strong restriction not expected even for general affine varieties. The second approach may appear more accessible in light of Cartan's Theorem~A; however, even constructing a single section satisfying the growth condition $\|s\|_\beta\le e^{A\phi+B}$ is in general highly nontrivial. 

Indeed, the approach of Aguilar and Burns ultimately amounts to finding a nonvanishing $(n,0)$-form with controlled growth, i.e., a special section of the canonical bundle. This in turn yields a nonvanishing holomorphic section of $K_X\otimes\overline{K_X}$ via $V^{n,0} \wedge \overline{V^{n,0}}$. However, $V^{n,0}$ may in general have poles; therefore, this construction provides only a single section in the second strategy described above. We refer to Section \ref{egtdisc} for further details.


In the context of this usual strategy, Theorem~\ref{main} provides a criterion that is more broadly applicable than Demailly's original criterion. Indeed, even in situations where it is difficult or impossible to construct more than one section $s$ of $K_X^{-m}$ with controlled growth arising from the second strategy, by setting $\psi := \frac{1}{m}\log \|s\|_\beta^2$, one still obtains a certain degree of algebraicity of $X$.

The statements in the theorem can be interpreted as follows.
\begin{itemize}
\item[(i)] \emph{Without any cohomological assumptions}, the existence of such a $\psi$ ensures that $X\setminus S$ is covered by birationally equivalent affine Zariski open subsets.
\item[(ii), (iii)] As additional reasonable conditions are imposed, stronger global algebraicity of $X \setminus S$ follows.
\item[(iv)] In the case $S = \emptyset$, one recovers Demailly's Affineness Criterion~\ref{DAC} under a relaxed regularity condition on $\psi$.
\end{itemize}

Finally, we remark that the additional assumptions appearing at each stage of Theorem~\ref{main} are, in a certain sense, essential. First, the finite-dimensionality of cohomology required to pass from (i) to (ii) is already noted to be necessary for global algebraicity in Demailly's original work \cite[Chapter~13]{D85}. For the additional assumption from (ii) to (iii), note that $\psi$ can be quite arbitrary. For instance, let $X = \mathbb{C}^n$ and $\phi = \log(1+\|Z\|^2)$. Then $\beta := dd^c e^\phi$ is Ricci-flat, and any plurisubharmonic function $\psi$ satisfying $\psi \le A\phi + B$ meets the condition of the theorem. If we take $\psi = n \log \|Z\|$, then $X \setminus S = \mathbb{C}^n \setminus \{0\}$, so all the conditions in (ii) are satisfied, but this space is not Stein for $n\ge 2$. This shows that the Stein assumption in (iii) cannot be omitted.

\noindent\textbf{Acknowledgements.}
The author is grateful to Prof. Chi Li for helpful discussions, which greatly contributed to the development of this paper. The author also thanks Prof. Burt Totaro, Prof. L\'aszl\'o Lempert, and Prof. R\'obert Sz\H oke for their kind interest and insightful comments. The author was partially supported by the NSF (Grant No.~DMS-2305296).

\part{Revised Demailly's Criterion}
Our main observations for the proof of Theorem \ref{main} are as follows, and the proof will proceed accordingly.

\begin{enumerate}
\item If $e^{-2\psi}$ is integrable, then most of the $L^2$ estimates essential for controlling the function theory on $X$ become available.

\item The singularity set
\[
S := \{ x \mid e^{-2\psi} \notin L^1_\loc(x) \}
\]
to which the above argument does not apply is a zero locus of a multiplier ideal. Hence, using functions $b_p$ obtained via suitable weighted $L^2$ estimates, one can ``cut out'' $S$ and reduce the situation to one analogous to Demailly's original criterion. More precisely, using weighted $L^2$ functions $b_p$ satisfying $b_p|_S = 0$, we consider the open set $X \setminus b_p^{-1}(0)$. This allows us to carry over much of the argument of Demailly's criterion to the present setting.

\item Moreover, the function theory on $X \setminus b_p^{-1}(0)$ enjoys many remarkable properties. By exploiting these, one can prove the affineness of $X \setminus b_p^{-1}(0)$ without imposing any cohomological conditions.
\end{enumerate}

\noindent\textbf{Conventions and notation.}
Throughout this paper, we define $d^c = i(\bar{\partial} - \partial)$ and $dd^c = 2i\,\partial\bar{\partial}$. For a real-valued function $V$, we write $V_+ := \max(0,V)$.
\section{Demailly's Monge--Amp\`ere measure theory}
\subsection{$\phi$-polynomials}

A fundamental difference between holomorphic functions on a Stein manifold and regular functions on an affine variety is that the former allow exponential growth (for instance $e^z$) and hence need not be algebraic. Consequently, in order to prove that a given Stein manifold is in fact affine, one must select candidates for coordinate functions whose growth is sufficiently controlled so that they behave similarly to regular functions.

We therefore introduce the following definitions from \cite{D85}.

\begin{defn}\label{phi}
Let $X$ be an $n$-dimensional Stein manifold and let $\phi$ be a strictly psh exhaustion function on $X$. Set $\alpha := dd^c\phi$.

\begin{enumerate}
\item For $r>0$, consider the hypersurface $S(r):=\{\phi=r\}$ equipped with the measure
\[
\mu_r := \alpha^{n-1}\wedge d^c\phi ,\qquad \mu_r(V):=\int_{S(r)}V d\mu_r~,
\]
where $V$ is a measurable function. For a measurable function $F$, define the $\phi$-degree of $F$ by
\[
\delta_\phi(F) := \limsup_{r\to\infty} \frac{1}{r} \int_{S(r)} \log_+|F| \, d\mu_r =\limsup_{r\to\infty}\frac{1}{r} \mu_r(\log_+|F|) \in [0,\infty],
\]
where $\log_+ t := \max(0,\log t)$.

\item Define
\[
A_\phi(X) := \{F\in\cO(X)\mid \delta_\phi(F)<\infty\}.
\]
Elements of $A_\phi(X)$ are called \emph{$\phi$-polynomials}. By the properties of $\limsup$ and $\log$, $A_\phi(X)$ is a ring.

\item The fraction field
\[
\cK_\phi(X):=\left\{\frac{F}{G}\mid F,G\in A_\phi(X), G\not\equiv0\right\}
\]
is called the \emph{$\phi$-rational function field}, and its elements are called \emph{$\phi$-rational functions}.
\end{enumerate}
\end{defn}

The following result is a consequence of \cite[Theorem 8.5]{D85}.

\begin{thm}\label{trdeg}
Let $X$ be an $n$-dimensional Stein manifold with finite Monge--Amp\`ere volume
\[
\Vol(X):=\int_X (dd^c\phi)^n < \infty.
\]
Then the following hold:
\begin{enumerate}
\item $0\le \trdeg_{\bC}\cK_\phi(X)\le n$.
\item If $\trdeg_{\bC}\cK_\phi(X)=n$, then $\cK_\phi(X)$ is a finite type extension of $\bC$.
\end{enumerate}
\end{thm}

\subsection{Weighted $\phi$-polynomials}

Let $X$ be a Stein manifold endowed with a strictly psh exhaustion $\phi$ and set
\[
\beta := dd^c(e^\phi).
\]
Since the main arguments to follow rely on weighted $L^p$ estimates (in particular $L^2$ estimates) with respect to the measure $\beta^n$, we introduce an analogue of $\phi$-polynomials adapted to this weighted setting, again from \cite{D85}.

\begin{defn}\label{weighted}
Define
\[
L_\phi^p(X):=\left\{f\in L^0(X)\;\middle|\; \int_X |f|^p e^{-C\phi}\beta^n<\infty \text{ for some } C>0 \right\},
\]
\[
L_\phi^\infty(X):=\left\{f\in L^0(X)\;\middle|\; |f|\le e^{C(1+\phi)} \text{ a.e. for some } C>0 \right\},
\]
\[
L_\phi^0(X):=\bigcup_{p>0} L_\phi^p(X), \qquad 
A_\phi^p(X):=L_\phi^p(X)\cap \cO(X).
\]
\end{defn}

The following properties follow from \cite[Lemmas 11.2 and 11.3]{D85}.

\begin{lem}\label{weightphi}
\begin{enumerate}
\item $1\in A_\phi^p(X)$ for all $p\ge 0$, and for $p\ge q>0$,
\[
L_\phi^p(X)\subset L_\phi^q(X), \qquad A_\phi^p(X)\subset A_\phi^q(X).
\]

\item $L_\phi^0(X)$ is a $\bC$-algebra, and $A_\phi^0(X)$ is an integrally closed subalgebra of $L_\phi^0(X)$.

\item If $V$ is plurisubharmonic and $e^V \in L_\phi^0(X)$, then $\delta_\phi(e^V) < \infty$. In particular, $A_\phi^0(X) \subset A_\phi(X)$ and elements of $A_\phi^0(X)$ are $\phi$-polynomials.

\item For $f\in A_\phi^p(X)$, one has $|df|_\beta\in L_\phi^p(X)$.
\end{enumerate}
\end{lem}

Note that in \cite[Lemma 11.3(a)]{D85}, Demailly proved that if a holomorphic function $f$ belongs to $L_\phi^0(X)$, then $\delta_\phi(f) < \infty$. However, the argument essentially proves Lemma \ref{weightphi}(3). This observation will be used later in Section \ref{algebraization} for the Siegel-type vanishing order counting argument of holomorphic functions. Moreover, statement (3) allows us to restrict future constructions to functions in $A_\phi^0(X)$ while retaining the key properties of $\phi$-polynomials.

\section{Construction of $b_p$ and $X\setminus b_p^{-1}(0)$}\label{cutout}
First, note that $S:=\{x \mid e^{-2\psi}\notin L^1_\loc(x)\}$ coincides with the zero set of the multiplier ideal
\[
\mathcal{J}(\psi)
=
\left\{
f \in \mathcal{O}(X)
\;\middle|\;
|f|^2 e^{-2\psi} \in L^1_\loc(X)
\right\}.
\]
In particular, the existence of $\psi$ yields functions in $\mathcal{O}(X)$ that vanish along $S$ and have controlled $L^p$ norm. More precisely, this follows from the following Bombieri--H\"ormander--Skoda type lemma.
\begin{lem}\label{skoda}
Let $\tau\in L^1_\loc(X)$ satisfy
\[
i\partial\bar\partial\tau+\Ric(\beta)\;\ge\;\lambda\,\beta
\]
for some positive continuous function $\lambda$.
Let $u$ be a $(0,1)$-form on $X$ with locally $L^2$ coefficients such that $\bar\partial u=0$ and
\[
\int_X \lambda^{-1}\,|u|_\beta^2\,e^{-\tau}\,\beta^n \;<\;\infty.
\]
Then there exists $g\in L^2_{\loc}(X)$ with $\bar\partial g=u$ and
\[
\int_X |g|^2 e^{-\tau}\beta^n \;\le\; \int_X \lambda^{-1}|u|_\beta^2 e^{-\tau}\beta^n \;<\;\infty.
\]
\end{lem}
Using this lemma, we control the function theory of $X$ as follows.
\begin{prop}\label{jet}
The following statements hold.
\begin{enumerate}
\item[(a)] Given finitely many points $p_1,\dots,p_k\notin S':=\{x~|~e^{-\psi}\notin L^1_\loc(x)\}\subset S$, and prescribed $s_i$-jets at each
$p_i$, there exists $f\in A_\phi^0(X)$ realizing the prescribed $s_i$-jet at each $p_i$.

\item[(b)] Let $f_1,\dots,f_n\in A_\phi^0(X)$ satisfy
\[
\bigwedge_{i=1}^n df_i(p)\neq 0
\qquad\text{for some }p\notin S.
\]
Then there exists $b_p\in A_\phi^0(X)$ such that $b_p(p)=1$ and
\[
\{\textstyle\bigwedge_{i=1}^n df_i=0\}\subset \{b_p=0\}.
\]

\item[(c)] $\trdeg_\mathbb{C}\cK_\phi(X)=n$, and $\cK_\phi(X)$ is a finite type extension of $\mathbb{C}$.
\end{enumerate}
\end{prop}

\begin{proof}
\noindent (a)
First, note that $S$ and $S'$ are closed. Hence, we can fix finitely many points $p_1,\dots,p_k$ and choose local coordinate systems
$z^{(1)},\dots,z^{(k)}$ near $p_1,\dots,p_k$, respectively, such that these coordinate neighborhoods are disjoint
from $S'$.
Let $\chi_i$ be a cutoff function supported in the coordinate neighborhood of $z^{(i)}$ and identically $1$
near $p_i$.
For each $p_i$, choose a polynomial $P_i$ of degree $s_i$ in $z^{(i)}$ whose $s_i$-jet at $p_i$ is the desired one.
Set
\[
h \;:=\; \sum_i \chi_i P_i,\qquad u\;:=\;\bar\partial h,\qquad s\;:=\;\max_i s_i.
\]
Choose $C>0$ sufficiently large so that
\[
\eta \;:=\; C\phi + (n+s)\sum_i \chi_i \log|z^{(i)}|^2
\]
is plurisubharmonic. Define
\[
\tau \;:=\; \psi+\eta+2\log(1+e^\phi).
\]
Then, since $\Ric(\beta)+\idd\psi\ge 0$ and $dd^c\log(1+e^\phi)\ge\frac{1}{(1+e^\phi)^2}\beta$, one obtains
\[
\idd\tau+\Ric(\beta)\;\ge\;\frac{1}{(1+e^\phi)^2}\,\beta.
\]
Hence we may take $\lambda=(1+e^\phi)^{-2}$.
Since $u:=\bar \partial h$ vanishes near $p_1,\cdots,p_k$ and $\supp(u)$ is relatively compact, while $e^{-\psi}$ is locally integrable on $\supp(u)$, we have
\[
\int_X \lambda^{-1}|u|_\beta^2 e^{-\tau}\beta^n <\infty.
\]
By the lemma there exists $g$ with $\bar\partial g=u$ and $\int_X |g|^2 e^{-\tau}\beta^n<\infty$.
Moreover, since $\psi\le A\phi+B$, $\eta\le C\phi+C'$ for some $C'>0$, and
$2\log(1+e^\phi)\le 2\phi+2\log 2$, there exist constants $A',B'$ such that
\[
e^{-\tau}\;\ge\; e^{-(A'\phi+B')}.
\]
Thus
\[
\int_X |g|^2 e^{-(A'\phi+B')}\beta^n <\infty,
\]
so $g\in L_\phi^0(X)$.
Finally, since $\bar\partial h=u$ and $\bar\partial g=u$, we have $\bar\partial(h-g)=0$, hence
\[
f\;:=\;h-g \in \mathcal{O}(X),
\qquad\text{and in particular } f\in A_\phi^0(X).
\]
The finiteness
\[
\int_X |g|^2 e^{-\tau}\beta^n
=\int_X |g|^2 e^{-\psi-C\phi}(1+e^\phi)^{-2}
\Bigl(\prod_i |z^{(i)}|^{-2(n+s)\chi_i}\Bigr)\,\beta^n \;<\;\infty
\]
forces $g$ to have no nonzero jets of order $\le s$ at each $p_i$.
Therefore $f=h-g$ has, at each $p_i$, the prescribed jet (matching that of $h$ up to order $s_i$).

\medskip
\noindent (b)
We apply the same argument as (a) with a slight modification.
Choose a local coordinate system $z$ near $p$ such that the neighborhood is disjoint from
$S$ and from $\{\bigwedge_{i=1}^n df_i=0\}$.
Let $\chi$ be a cutoff supported in this coordinate neighborhood and identically $1$ near $p$.
Set $h:=\chi$ and
\[
\eta \;:=\; C\phi + n\,\chi \log|z|^2,
\]
for some $C>0$ chosen so that $\eta$ is psh.
Define
\[
\tau \;:=\; 2\psi+\eta+\log\bigl|\textstyle\bigwedge_{i=1}^n df_i\bigr|_\beta^2 + 2\log(1+e^\phi).
\]
Let $[Z]$ denote the zero divisor of $\bigwedge_{i=1}^n df_i$ (viewed as a $(1,1)$-current).
Then
\[
\idd\tau+\Ric(\beta)
=2\idd\psi+2\Ric(\beta)+2\pi[Z]+\idd\eta+dd^c\log(1+e^\phi)
\;\ge\;\frac{1}{(1+e^\phi)^2}\,\beta.
\]
As in the proof of (a), $\bar\partial h$ vanishes near $p$ and $\supp(\bar\partial h)$ is relatively compact. Moreover, on $\supp(\bar\partial h)$, $e^{-2\psi}$ is locally integrable and $\log\bigl|\textstyle\bigwedge_{i=1}^n df_i\bigr|_\beta^2$ is finite.  Therefore, taking again $\lambda=(1+e^\phi)^{-2}$ yields
\[
\int_X \lambda^{-1}|\bar\partial h|_\beta^2 e^{-\tau}\beta^n<\infty.
\]
By the lemma there exists $g$ with $\bar\partial g=\bar\partial h$ and
\[
\int_X |g|^2 e^{-\tau}\beta^n
=\int_X |g|^2 e^{-2\psi-C\phi}\,|z|^{-2n\chi}\,
\bigl|\textstyle\bigwedge_{i=1}^n df_i\bigr|_\beta^{-2}\,(1+e^\phi)^{-2}\,\beta^n
\;<\;\infty.
\]
The finiteness of this integral, despite the factors
$\bigl|\bigwedge_{i=1}^n df_i\bigr|_\beta^{-2}$ and $|z|^{-2n\chi}$, forces $g$ to vanish on
$\{\bigwedge_{i=1}^n df_i=0\}$ and at $p$.
Define $b_p:=h-g$. Then $\bar\partial b_p=0$, so $b_p$ is holomorphic, $b_p(p)=1$, and
$\{\bigwedge_{i=1}^n df_i=0\}\subset\{b_p=0\}$.
Moreover, since $\psi\le A\phi+B$, $\eta\le C\phi+C'$ for some $C'>0$, and
$2\log(1+e^\phi)\le 2\phi+2\log 2$, there exist $A',B'$ such that
\[
e^{-\tau}\;\ge\; e^{-(A'\phi+B')}\,\bigl|\textstyle\bigwedge_{i=1}^n df_i\bigr|_\beta^{-2}.
\]
Hence $|g|\cdot\bigl|\bigwedge_{i=1}^n df_i\bigr|_\beta^{-1}\in L_\phi^2(X)$, and therefore
\[
|g|\;\le\;\Bigl(|g|\cdot\bigl|\textstyle\bigwedge_{i=1}^n df_i\bigr|_\beta^{-1}\Bigr)
\prod_{i=1}^n |df_i|_\beta \in L_\phi^0(X),
\]
which implies $b_p=h-g\in A_\phi^0(X)$.

\medskip
\noindent (c)
By (a), we can construct $n$ functions $f_1, \cdots, f_n$ which separate all $1$-jets at a point $p\notin S$. These functions must be algebraically independent, hence $\trdeg_{\mathbb{C}} \cK_\phi(X) \ge n$. By Theorem~\ref{trdeg}, it follows that $\trdeg_{\mathbb{C}} \cK_\phi(X) = n$, and $\cK_\phi(X)$ is a finite type extension of $\mathbb{C}$.
\end{proof}

\begin{rem}\label{vanish}
The function $b_p$ produced in the above Proposition \ref{jet} vanishes on $S:=\{x~|~e^{-2\psi}\notin L^1_\loc(x)\}$.
Indeed, $b_p$ is of the form $h-g$, where $h$ vanishes on $S$ by construction, and where the finiteness
of an integral of the form $\int_X |g|^2 e^{-\tau}\beta^n<\infty$ forces $g$ to vanish on $S$, since
$e^{-\tau}\sim e^{-2\psi}$ along $S$.
\end{rem}

\begin{prop}\label{isom}
For any point $p\notin S:=\{x~|~e^{-2\psi}\notin L^1_\loc(x)\}$, there exists a function $b_p\in A_\phi^0(X)$ with $b_p(p)=1$
and functions $f_1,\dots,f_N\in A_\phi^0(X)$ such that
\[
F\;:=\;(b_p,f_1,\dots,f_N)
\]
induces an analytic isomorphism from $X\setminus b_p^{-1}(0)$ onto a smooth open subset of an $n$-dimensional irreducible
normal algebraic variety $\cA\subset\mathbb{C}^{N+1}$.
\end{prop}

\begin{proof}
By the previous proposition~\ref{jet}, choose $f_1,\dots,f_n\in A_\phi^0(X)$ separating $1$-jets at $p$, i.e.
$\bigwedge_{i=1}^n df_i(p)\neq 0$, and choose $b_p\in A_\phi^0(X)$ such that $b_p(p)=1$ and
$b_p$ vanishes on $\{\bigwedge_{i=1}^n df_i=0\}$.
Define inductively
\[
F_{j}:=(b_p,f_1,\dots,f_n,f_{n+1},\dots,f_{n_j}),\qquad f_i\in A_\phi^0(X),
\]
with $F_{0}:=(b_p,f_1,\dots,f_n)$, and let $\cA_j$ be the Zariski closure of $F_j(X)$ in $\mathbb{C}^{n_j+1}$.
\begin{enumerate}
\item[(i)] If $\cA_j$ is not normal, let $\cA_{j+1}$ be the normalization of $\cA_{j}$ and define
$F_{j+1}=(b_p,f_1,\dots,f_{n_{j+1}})$ to be the lift of $F_{j}$ to $\cA_{j+1}$.

\item[(ii)] If $\cA_{j}$ is normal but there exist distinct $z,z'\in X\setminus b_p^{-1}(0)$ with
$F_{j}(z)=F_{j}(z')$, then by Proposition \ref{jet} choose $f_{n_j+1}\in A_\phi^0(X)$ such that
$f_{n_j+1}(z)\neq f_{n_j+1}(z')$, and set
\[
F_{j+1}:=(b_p,f_1,\dots,f_n,f_{n+1},\dots,f_{n_j},f_{n_j+1}).
\]
\end{enumerate}

Step (i) is well-defined: any newly added coordinate function in the normalization is integral over the subring generated by the existing coordinates, and since $A_\phi^0(X)$ is integrally closed, it belongs to $A_\phi^0(X)$. Step (ii) is also well-defined. Since $b_p^{-1}(0)$ contains $S$ (Remark \ref{vanish}), it contains $S' := \{x \mid e^{-\psi}\notin L^1_\loc(x)\} \subset S$, which is the locus where Proposition~\ref{jet} (a) cannot be applied. Consequently, on $X \setminus b_p^{-1}(0)$ one may freely apply Proposition~\ref{jet} (a).

Since $f_1,\dots,f_n$ occur among the entries of $F_{j}$ and $\trdeg_{\mathbb{C}}\cK_\phi(X)=n$ (Proposition \ref{jet}, (c)),
each $\cA_{j}$ is an $n$-dimensional irreducible algebraic variety.
Moreover, on $X\setminus b_p^{-1}(0)$ we have $\bigwedge_{i=1}^n df_i\neq 0$, hence $F_{j}$ is \'{e}tale there.

The iteration of (i) and (ii) terminates after finitely many steps.
Indeed, suppose that there exist distinct points $z\neq z'$ such that $F_j(z)=F_j(z')=:m\in \cA_j$ and $f_{n_j+1}(z)\neq f_{n_j+1}(z')$. Since $F_j$ is \'etale at both $z$ and $z'$, there exists a sufficiently small neighborhood $W$ of $m$ such that $F_j$ is biholomorphic onto $W$ on two branches $Z\ni z$ and $Z'\ni z'$, respectively. If $f_{n_j+1}\in \bC(b_p,f_1,\dots,f_{n_j})$, then, since $\bC(b_p,f_1,\dots,f_{n_j})$ is the rational function field of $\cA_j$, there exists a rational function $R$ on $\cA_j$ such that $f_{n_j+1}=R\circ F_j$. This contradicts the fact that $F_j(z)=F_j(z')$ but $f_{n_j+1}(z)\neq f_{n_j+1}(z')$. Therefore,
\[
\bigl[\bC(b_p,f_1,\dots,f_{n_j},f_{n_j+1}):\bC(b_p,f_1,\dots,f_{n_j})\bigr]\ge 2.
\]
Since $\cK_\phi(X)$ is a finite type extension of $\mathbb{C}$, this can occur only finitely many times. Thus, after finitely many steps, we obtain the desired $F$ and $\cA$.
\end{proof}

\begin{rem}\label{addition}
Note that in the construction of Proposition~\ref{isom}, if one arbitrarily adds an element $f \in A_\phi^0(X)$ to the entries of $F$ and then performs again the normalization step~(i), the conclusion of Proposition~\ref{isom} remains unaffected. This observation will be used later when we construct $F$ for different choices of $b_p$ in the proof of Theorem~\ref{main}~(ii).
\end{rem}

\section{Algebraization of $X\setminus b_p^{-1}(0)$}\label{algebraization}

We now prove that the map
\[
F: X\setminus b_p^{-1}(0)\longrightarrow \cA\subset \bC^{N+1}
\]
constructed in Proposition \ref{isom} sends $X\setminus b_p^{-1}(0)$ biholomorphically onto an affine subvariety of $\cA$. The argument consists of two main steps.

\begin{enumerate}
\item After slight modification of $b_p,F,\cA$ into $f_p,\check F,\check \cA$, we prove the existence of a meromorphic rational $(1,0)$-form $h$ on $\check \cA$ whose pole divisor coincides with $\check \cA\setminus \check F\bigl(X\setminus f_p^{-1}(0)\bigr)$. This implies that $\check F(X\setminus f_p^{-1}(0))$ is quasi-affine. We then prove that $\cA\setminus F(X\setminus b_p^{-1}(0))$ is an algebraic hypersurface of $\cA$.

\item We show that $F\bigl(X\setminus b_p^{-1}(0)\bigr)$ is rationally convex, i.e., for any compact set 
$K\subset F\bigl(X\setminus b_p^{-1}(0)\bigr)$ and the regular function ring 
$R\bigl(F(X\setminus b_p^{-1}(0))\bigr)$, the rational convex hull 
\[
\widehat K
=\left\{x\in F\bigl(X\setminus b_p^{-1}(0)\bigr)\;\middle|\;
|f(x)|\le \sup_K |f| \text{ for all } f\in R\bigl(F(X\setminus b_p^{-1}(0))\bigr)\right\}
\]
is a compact subset of $F(X\setminus b_p^{-1}(0))$. Together with a theorem of Mok to be recalled later, this will imply that 
$F\bigl(X\setminus b_p^{-1}(0)\bigr)$ is affine.
\end{enumerate}

\subsection{$X\setminus b_p^{-1}(0)$ is quasi-affine}
First, we need the following modification of $b_p$, $F$, and $\cA$ as follows.
\begin{con}\label{con}
Let $F:X\to\mathbb{C}^{N+1}$, $F=(b_p,f_1,\dots,f_N)$, and $\cA$ be the same as Proposition \ref{isom}.
Choose a polynomial $Q$ on $\mathbb{C}^{N+1}$ which vanishes on the singular locus of $\cA$, $Q\circ F(p)=1$, and is divisible
by $z_0$ so that $b_p~|~Q\circ F$.
Define
\[
f_p \;:=\; Q\circ F\in A_\phi^0(X),\qquad
\check \cA:=\{(z,w)\in \mathbb{C}^{N+1}\times\mathbb{C} \;|\;
z\in \cA,\; Q(z)w=1\},\]
\[
\check F: X\setminus f_p^{-1}(0)\to \check\cA\subset \mathbb{C}^{N+2},\qquad
\check F \;:=\; \Bigl(F,\frac{1}{f_p}\Bigr).
\]
\end{con}
$F$, $b_p$, and Construction~\ref{con} are intentionally designed to satisfy the following two properties. (1) All entries of $F$ and $f_p$ belong to $A_\phi^0(X)$. (2) The zero set of $f_p$ absorbs $S:=\{x~|~e^{-2\psi}\notin L^1_\loc(x)\}$ (Remark \ref{vanish}). This design makes $\check F\!\left(X\setminus f_p^{-1}(0)\right)$ almost the same as the setting of Demailly's original criterion. Consequently, several estimates from the original proof of Demailly's affineness criterion \cite{D85} also apply in the present setting. Whenever this occurs, we note why the corresponding estimate from \cite{D85} remains valid here and apply it to avoid excessive repetition.




\subsubsection{Construction of the $(1,0)$-form $h$}

We begin by constructing a candidate $(1,0)$-form $h$ in step 1. The following theorem is the main ingredient (\cite{D85}, Theorem 15.3).

\begin{thm}\label{mingrowth}
Let $\cA$ be an $n$-dimensional smooth affine algebraic subvariety of $\bC^N$ and set
\[
\omega := dd^c\log(1+|z|^2).
\]
Assume that there exists a positive closed $(1,1)$-current $T$ on $\cA$ such that
\[
\int_\cA T\wedge \omega^{n-1} < \infty.
\]
Then there exist a psh function $V$ on $\cA$ and a $C^\infty$ $(1,0)$-form $u$ on $\cA$ such that for some constants $C_1,C_2,C_3>0$:
\begin{enumerate}
\item $dd^c V \ge T$,
\item $V(z)\le C_1\log^+|z|$,
\item $dd^c V - T = \bar\partial u$,
\item $|u|_\omega \le C_2(1+|z|)^{C_3}$.
\end{enumerate}
\end{thm}

Using this theorem, we define five objects: $T$, $V$, $u$, $\tau$, and ultimately the desired $(1,0)$-form $h$. Define a $(1,1)$-current $T$ on $\check \cA$ by
\[
T := dd^c(\phi\circ \check F^{-1})\quad \text{on } \check F\bigl(X\setminus f_p^{-1}(0)\bigr), 
\qquad 
T := 0\quad \text{on } \check \cA\setminus \check F\bigl(X\setminus f_p^{-1}(0)\bigr).
\]
Since all entries of $F$ and $f_p$ belong to $A_\phi^0(X)$, we may apply \cite{D85}, Proposition~12.2, to conclude that $T$ satisfies the condition of Theorem~\ref{mingrowth}.
Let $V$ be the plurisubharmonic function and $u$ the $(1,0)$-form given by the theorem, associated to the current $T$ and the form
\[
\omega := dd^c\log(1+|z|^2)\quad \text{on } \bC^{N+2}.
\]

On $\check F\bigl(X\setminus f_p^{-1}(0)\bigr)$, define
\[
\tau := V - \phi\circ \check F^{-1}.
\]
By Theorem \ref{mingrowth}.1., $\tau$ is psh and $\tau\le V$. Since $\phi$ is an exhaustion and the last coordinate of $\check F$ is $f_p^{-1}$, $\phi\circ \check F^{-1}$ diverges to $+\infty$ near the boundary of $\check F\bigl(X\setminus f_p^{-1}(0)\bigr)$. Hence $\tau$ diverges to $-\infty$ along the boundary. Therefore $\tau$ extends to a psh function on all of $\check \cA$ that is identically $-\infty$ on
\[
\check \cA\setminus \check F\bigl(X\setminus f_p^{-1}(0)\bigr).
\]
We obtain:

\begin{prop}\label{pluripolar}
The set $\check \cA\setminus \check F\bigl(X\setminus f_p^{-1}(0)\bigr)$ is pluripolar.
\end{prop}

Finally, define on $\check F\bigl(X\setminus f_p^{-1}(0)\bigr)$
\[
h := \partial\tau - \frac{i}{2}\,u.
\]
By Theorem \ref{mingrowth}.3., we have $2i\partial\bar\partial\tau=\bar\partial u$ on 
$\check F\bigl(X\setminus f_p^{-1}(0)\bigr)$, and since $u$ is smooth on this domain, it follows that $h$ is holomorphic on $\check F\bigl(X\setminus f_p^{-1}(0)\bigr)$.

\subsubsection{Properties of $h$}
For $\check F:=(b_p,f_1,\cdots,f_N,1/f_p)$, since $df_1,\cdots,df_n$ are linearly independent on $X\setminus f_p^{-1}(0)$, any $(1,0)$-form $h$ on $\check F(X\setminus f_p^{-1}(0))$ can be written in the form
\[
h=\sum_{j=1}^{n} h_j\,dz_j,
\]
where each $h_j$ is a holomorphic function on $\check F(X\setminus f_p^{-1}(0))$.

We now prove the following two statements.
\begin{prop}\label{rational}
Each $h_j$ is a rational function on $\check \cA$. In particular, $h$ is a rational meromorphic $(1,0)$-form on $\check \cA$.
\end{prop}
\begin{prop}\label{quasi-affine}
The pole divisor of $h$ is precisely $\check \cA \setminus \check F(X \setminus f_p^{-1}(0))$. In particular, $\check F(X \setminus f_p^{-1}(0))$ is quasi-affine.
\end{prop}

First fix one of the $h_j$. Since all entries of $F$ and $f_p$ belong to $A_\phi^0(X)$, and by the proof of Proposition \ref{jet} (b) used in the construction of $b_p$ we have 
\[
|b_p|\cdot |\wedge_{i=1}^n df_i|_\beta^{-1}\in L_\phi^0(X),
\]
it follows from \cite{D85}, Proposition 12.6 (c), that the following holds (note that $b_p$ corresponds to the auxiliary function denoted by $f_{n+1}$ in the proof of \cite{D85}). Analogous to $L_\phi^0(X)$, define
\[
L_\phi^0(X\setminus f_p^{-1}(0)):=\left\{f\;\middle|\; \int_{X\setminus f_p^{-1}(0)} |f|^p e^{-C\phi}\beta^n<\infty \text{ for some } C>0 \text{ and } p>0\right\}.
\]
\begin{prop}\label{bound}
There exists a sufficiently large integer $s>0$ such that for $g:=f_p^s h_j$ we have
\[
e^{\tau\circ\check F}|g|\in L_\phi^0(X\setminus f_p^{-1}(0)).
\]
\end{prop}

We also recall the following lemma from \cite[Lemme 12.8]{D85}.

\begin{lem}\label{prolong}
Let $X$ be a complex manifold and $g$ a holomorphic function on $X$. Let $\theta$ be a plurisubharmonic function on $X\setminus g^{-1}(0)$. If $e^\theta$ is locally integrable, then $\theta+\log|g|^2$ extends to a plurisubharmonic function on all of $X$.
\end{lem}

Let $P(x_0,\cdots,x_n)$ be a polynomial in $n+1$ variables whose degree in each variable is at most $k$, and consider $P(g,f_1,\cdots,f_n)$ for the function $g$ in Proposition \ref{bound}. By the proposition and Lemma \ref{weightphi}, we have
\[
e^{k\tau\circ\check F}P(g,f_1,\cdots,f_n)\in L_\phi^0(X\setminus f_p^{-1}(0)),
\]
hence it is locally integrable. Therefore, by Lemma \ref{prolong},
\[
\log|P(g,f_1,\cdots,f_n)|+k\,\tau\circ\check F+\log|f_p|^2
\]
extends from $X\setminus f_p^{-1}(0)$ to a plurisubharmonic function on all of $X$.

Recall also the definition of the $\phi$-degree $\delta_\phi$ introduced in Definition \ref{phi} and Lemma \ref{weightphi} (3): for plurisubharmonic $V$, if $e^V\in L_\phi^0(X)$, then $\delta_\phi(e^V)<\infty$. Hence $\delta_\phi(e^{\tau\circ\check F}g)<C$. We also note that, by definition of $\delta_\phi$, one has $\delta_\phi(fg) \le \delta_\phi(f) + \delta_\phi(g)$. Therefore, 
\[
\delta_\phi\left(e^{k\tau\circ\check F}|P(g,f_1,\cdots,f_n)||f_p|^2\right)<Ck+C'
\]
for some constant $C'>0$.

We also recall the following results from \cite{D85}, Corollaire 7.3, Lemme 7.4, Proposition 8.3, and Corollaire 8.4.

\begin{prop}\label{zero}
For a plurisubharmonic function $V$ and a holomorphic function $f$, the vanishing order of $f$ at a point $a$, denoted $\mathrm{ord}_a(f)$, satisfies
\[
\mathrm{ord}_a(f)<C(a)\,\delta_\phi(e^V f)
\]
for some constant $C(a)$ depending only on $a$.
\end{prop}

We now show that there exists a polynomial $P$ such that
\[
P(g,f_1,\cdots,f_n)=0.
\]
Suppose that no such polynomial exists. Let $\mathbb{P}$ be the vector space of polynomials in $n+1$ variables whose degree in each variable is at most $k$. Define a linear map
\[
J_s:\mathbb{P}\longrightarrow \mathbb{C}^{\binom{n+s}{n}}
\]
by sending $P$ to the tuple consisting of all jets of $P(g,f_1,\cdots,f_n)$ at the point $a$ of order $\le s$.

If
\begin{equation}\label{dimension}
\dim_{\mathbb{C}}\mathbb{P}=(k+1)^{n+1}>\binom{n+s}{n},
\end{equation}
then $\ker J_s\neq 0$. For $P\in\ker J_s$ we have
\[
\mathrm{ord}_a(P(g,f_1,\cdots,f_n))>s.
\]

However, for $a\in X\setminus f_p^{-1}(0)$, by Proposition \ref{zero} we must have
\[
s<\mathrm{ord}_a(P(g,f_1,\cdots,f_n))
<
C(a)\,\delta_\phi\!\left(e^{k\tau\circ\check F}P(g,f_1,\cdots,f_n)|f_p|^2\right)
<
C(a)(Ck+C').
\]

Since
\[
\binom{n+s}{n}\le \frac{(n+s)^n}{n!},
\]
for sufficiently large $k$ we can choose $s$ large enough so that $s>C(a)(Ck+C')$ while still satisfying \eqref{dimension}. This yields a contradiction. Hence for such an $s$ any $P\in\ker J_s$ must satisfy
\[
P(g,f_1,\cdots,f_n)=0.
\]

Finally, we use the following lemma to conclude that $g:=f_p^s h_j$ is a rational function on $\cA$, and hence $h_j$ is also rational.

\begin{lem}\label{rationalext}
Let $\cA$ be a normal affine variety and $\Omega$ a domain such that $\cA\setminus\Omega$ is a closed pluripolar set. Suppose $f\in\mathcal{O}(\Omega)$ is algebraic over $K(\cA)$, i.e.\ there exists a polynomial $P$ with coefficients in $K(\cA)$ such that $P(f)=0$. Then $f$ extends to a rational function on $\cA$.
\end{lem}

\begin{proof}
Since $f$ is algebraic over $K(\cA)$, multiplying by a suitable regular function we obtain
\[
\sum_{k=0}^{d} a_k f^k=0,\qquad a_k\in R(\cA),\; a_d\neq 0.
\]
Hence the function $g=a_d f$ satisfies the monic polynomial equation with coefficients in $R(\cA)$
\[
g^d+\sum_{k=0}^{d-1} a_d^{\,d-1-k}a_k g^k=0.
\]
Fix an affine embedding $\cA\subset \bC^N$, and let $z$ denote the standard coordinates on $\bC^N$. Then, by the Cauchy root bound applied to the above polynomial relation, there exist constants $C_1,C_2>0$ such that
\[
|a_d f|\le C_1(1+|z|)^{C_2}.
\]
In particular, $a_d f$ is locally bounded on $\Omega$. Since a locally bounded holomorphic function extends holomorphically across a closed pluripolar set, $a_d f$ extends holomorphically to $\cA_{\mathrm{reg}}$. By the normality of $\cA$, it then extends holomorphically to all of $\cA$.

Finally, any holomorphic function on an affine variety $\cA\subset \bC^N$ satisfying such a polynomial growth bound extends to a polynomial on $\bC^N$ \cite[Introduction]{B74}. Therefore $a_d f\in R(\cA)$, and hence $f\in K(\cA)$.
\end{proof}
Now we prove Proposition~\ref{quasi-affine}. Let $\Omega \subset \check \cA$ be the maximal domain where $h$ is holomorphic. First, $\check F(X \setminus f_p^{-1}(0)) \subset \Omega$ and recall that $\check \cA\setminus \check F(X \setminus f_p^{-1}(0))$ is pluripolar (Proposition \ref{pluripolar}). For $b \in \Omega$, there exist $a \in \check F(X \setminus f_p^{-1}(0))$ and a path $\gamma$ from $a$ to $b$ in $\Omega$ such that a simply connected tubular neighborhood $\Gamma \subset \Omega$ of $\gamma$ exists. Since $h=\partial \tau - \frac{i}{2}u$ and $u$ is smooth on $\check \cA$, setting $v:=h+\frac{i}{2}u$ gives a smooth form on $\Omega$ satisfying $v=\partial \tau$ on $\check F(X \setminus f_p^{-1}(0))$. On $\check F(X \setminus f_p^{-1}(0))$ we have $d(v+\bar v)=d(\partial \tau+\bar\partial \tau)=d\circ d \tau=0$, hence the same holds on $\Gamma$. Because $\Gamma$ is simply connected, there exists $\tau'$ with $d\tau'=v+\bar v$ on $\Gamma$, and $\tau'=\tau$ on $\check F(X \setminus f_p^{-1}(0))$. Thus $\tau$ extends to a finite-valued plurisubharmonic function on $\Omega$. However, as mentioned in the proof of Proposition~\ref{pluripolar}, $\tau$ must be $-\infty$ outside $\check F(X \setminus f_p^{-1}(0))$. Therefore $\Omega=\check F(X \setminus f_p^{-1}(0))$.

Finally, we prove the following.
\begin{prop}\label{complement}
$\cA \setminus F(X \setminus b_p^{-1}(0))$ is an algebraic hypersurface of $\cA$.
\end{prop}

\begin{proof}
First note that $\check \cA$ is a Zariski open subset of $\cA$. Moreover, $\check F(X \setminus f_p^{-1}(0))$ is a Zariski open subset of $\check \cA$ by Proposition~\ref{quasi-affine}. Since, for every $x\in X\setminus b_p^{-1}(0)$, we may choose $f_p$ such that $f_p(x)\neq 0$, it follows that $F(X\setminus b_p^{-1}(0))$ corresponds to a Zariski open subset of $\cA$.

Recall from Proposition~\ref{isom} that $\cA$ is normal and that $F(X\setminus b_p^{-1}(0))$ is Stein. Suppose that a point of the Zariski closed set $\cA\setminus F(X\setminus b_p^{-1}(0))$ has codimension $\ge 2$ in $\cA$. Then, by the Hartogs extension theorem for normal complex spaces, such a point must belong to the holomorphic convex hull of $F(X\setminus b_p^{-1}(0))$, which contradicts the Stein property. Therefore $\cA\setminus F(X\setminus b_p^{-1}(0))$ is a purely codimension~$1$ Zariski closed set.
\end{proof}

\subsection{$X \setminus b_p^{-1}(0)$ is affine}
By Proposition \ref{complement}, we may apply the following theorem of Mok \cite{M84} to $F(X\setminus b_p^{-1}(0))$.
\begin{thm}\label{Mok}
Let $\cA \subset \mathbb{C}^N$ be an $n$-dimensional affine variety not necessarily smooth, and let $H$ be an algebraic hypersurface of $\cA$. Suppose that $\cA \setminus H$ is rationally convex; that is, for any compact $K \subset \cA \setminus H$, the rational convex hull
\[
\widehat{K}:=\{x\in \cA \setminus H \mid |g(x)| \le \sup_K |g| \text{ for all } g \in R(\cA \setminus H)\}
\]
is also compact. Then $\cA \setminus H$ is an affine variety.
\end{thm}

To show that $F(X \setminus b_p^{-1}(0))$ is rationally convex, we first prove the following.

\begin{prop}\label{dense}
The following statements hold.
\begin{enumerate}
    \item $A_\phi^0(X)\!\left[\frac{1}{b_p}\right]$ is dense in $\cO(X)$ with respect to the topology of uniform convergence on compact subsets $K \subset X \setminus b_p^{-1}(0)$.
    \item If $S':=\{x~|~e^{-\psi}\notin L^1_\loc(x)\}=\emptyset$, then $A_\phi^0(X)$ is dense in $\cO(X)$ with respect to the topology of uniform convergence on compact subsets $K \subset X$.
\end{enumerate}
 
\end{prop}

\begin{proof}
Fix $f \in \cO(X)$, compact $K\subset X\backslash b_p^{-1}(0)$, and an integer $k > \sup_{x\in K}\phi(x)$. Choose a cutoff function $\chi$ such that $\chi=1$ on $\{\phi \le k+1\}$ and $\chi=0$ on $\{\phi \ge k+2\}$. Set
\[
h := \chi b_p f, \qquad u := \bar\partial h = (\bar\partial \chi)\, b_p f.
\]
Let $\eta_t := \max(t(\phi-k),0)$ and $\tau_t := \psi + \eta_t + 2\log(1+e^\phi)$. Then for $\lambda := (1+e^\phi)^{-2}$ we have
\[
\idd\tau_t + \Ric(\beta)
= \idd\psi + \Ric(\beta) + \idd\eta_t + 2\idd\log(1+e^\phi)
\ge \lambda \beta .
\]
Hence
\[
\int_X \lambda^{-1}|u|_\beta^2 e^{-\tau_t}\,\beta^n
= \int_X |\bar\partial\chi|_\beta^2 |b_p f|^2 e^{-\psi} e^{-\eta_t}\,\beta^n .
\]
Here $\bar\partial\chi$ is nonzero only on $\{k+1\le \phi \le k+2\}$, and on this region $|\bar\partial\chi|_\beta$ and $|f|$ are bounded, while $\eta_t\ge t$. Thus for some constant $C$,
\[
\int_X |\bar\partial\chi|_\beta^2 |b_p f|^2 e^{-\psi} e^{-\eta_t}\,\beta^n\le C e^{-t} \int_{k+1\le \phi \le k+2} |b_p|^2 e^{-\psi}\,\beta^n .
\]
From Proposition \ref{jet}, writing $b_p:=h-g$, the function $h$ vanishes near $S:=\{x~|~e^{-2\psi}\notin L^1_\loc(x)\}$, and for $g$ the squared integral with weight larger than $e^{-\psi}$ near $S$ is finite. Hence
\[
\int_{k+1\le \phi \le k+2} |b_p|^2 e^{-\psi}\,\beta^n
\]
is finite. Therefore, by Lemma \ref{skoda}, there exists $g_t$ satisfying $\bar\partial g_t=u$ and for some constant $C'$,
\[
\int_X |g_t|^2 e^{-\tau_t}\,\beta^n \le \int_X \lambda^{-1}|u|_\beta^2 e^{-\tau_t}\,\beta^n < C' e^{-t}.
\]
On $K$, we have $\tau_t=\psi+2\log(1+e^\phi)$, and since $\bar\partial g_t=u=\bar\partial (b_p f)$ vanishes on a neighborhood of $K$, the function $|g_t|^2$ is plurisubharmonic there. Since $K$ is compact, by the Cauchy estimate, for an open $U$ which is $K\subset \overline U\subset \{\phi < k\}\setminus b_p^{-1}(0)$, there exists a constant $C''$ such that
\[
\sup_K |g_t|^2
\le C'' \int_U |g_t|^2 e^{-\tau_t}\,\beta^n
\le C'' \int_X |g_t|^2 e^{-\tau_t}\,\beta^n
< C''C' e^{-t}.
\]
Set $f_t:=h-g_t$. Then $\bar\partial f_t=0$, and since $\tau_t \le A'\phi+B'$ for some $A',B'>0$, we have $g_t\in L_\phi^0(X)$ and $f_t\in A_\phi^0(X)$. Finally, on $K$ we have $h=b_p f$ and $b_p\ne 0$, hence for any $\varepsilon>0$ we can take $t$ sufficiently large so that on $K$,
\[
\left| f-\frac{f_t}{b_p}\right|
= \left|\frac{b_p f-f_t}{b_p}\right|
= \left|\frac{g_t}{b_p}\right|
<\varepsilon .
\]

If $S'=\emptyset$, the same proof with $b_p=1$ yields the desired conclusion.
\end{proof}

\begin{prop}\label{main1}
$F(X \setminus b_p^{-1}(0))$ is rationally convex. In particular, $F(X \setminus b_p^{-1}(0))$ is affine.
\end{prop}

\begin{proof}
First, note that the first coordinate of $F$ is $b_p$. Hence $1/b_p \circ F^{-1}$ is a regular function on $F(X \setminus b_p^{-1}(0))$. Moreover, for any $f \in A_\phi^0(X)$, the function $f$ is algebraic over $\mathbb{C}(f_1,\cdots,f_n)$. Hence $f\circ F^{-1}$ is algebraic over $K(\cA)$, and Lemma \ref{rationalext} extends $f\circ F^{-1}$ to a rational function on $\cA$. On $F(X\setminus b_p^{-1}(0))$, $f\circ F^{-1}$ is smooth, hence it is a regular function on $F(X\setminus b_p^{-1}(0))$. 

Let $K \subset X \setminus b_p^{-1}(0)$ be a compact set and consider the rational convex hull $\widehat{F(K)}$ of $F(K)$, and $\widehat{K}:=F^{-1}\left(\widehat{F(K)}\right)$. Note that $\widehat{K}$ and $\widehat{F(K)}$ are homeomorphic. The set
\[
A:=\left\{x\in X \setminus b_p^{-1}(0) \,\middle|\, \frac{1}{|b_p(x)|} \le \sup_K \frac{1}{|b_p|} \right\}
\]
is closed and well-defined since $b_p$ does not vanish on $X \setminus b_p^{-1}(0)$, and by definition $\widehat{K} \subset A\subset X\setminus b_p^{-1}(0)$. Next, consider the holomorphically convex hull of $K$ in $X$,
\[
\overline{K}:=\{x \in X \mid |f(x)| \le \sup_K |f| \text{ for all } f \in \mathcal{O}(X)\}.
\]
Since $X$ is Stein, this set is compact. For $x \in A\subset X\setminus b_p^{-1}(0)$, if $x \notin \overline{K}$, then there exists $h \in \mathcal{O}(X)$ such that $|h(x)| > \sup_K |h|$. Applying Proposition \ref{dense} to $h$ on the compact set $K \cup \{x\}$, there exists $\tilde{h} \in A_\phi^0(X)[\frac{1}{b_p}]$ such that $|\tilde{h}(x)| > \sup_K |\tilde{h}|$. Since $\tilde h\circ F^{-1}$ is a regular function on $F(X\setminus b_p^{-1}(0))$, $x\notin\widehat{K}$. Hence it follows that
\[
\widehat{K} \subset \overline{K} \cap A.
\]
Since $\widehat{K}$ is closed by definition and $\overline{K} \cap A$ is compact, it follows that $\widehat{K}$ is compact, which proves the claim.
\end{proof}

\section{Global Aspects of Theorem~\ref{main}}
Finally, we prove the remaining global statements of Theorem~\ref{main}. Let $p_i\in X\setminus S$ be arbitrary points and set $b_i:=b_{p_i}$. Consider Remark~\ref{addition}. Namely, if we take the finite product $F^{(k)}:=F_1\times\cdots\times F_k$ and let $\cA^{(k)}$ be the Zariski closure of the image of $X$ under this map, then each $X\setminus b_i^{-1}(0)$ is mapped by $F^{(k)}$ onto a Zariski open subset of $\cA^{(k)}$, which completes the proof of the first statement.

The statements (ii), (iii), and (iv) of Theorem~\ref{main} are applications of part (i) proved in the previous section.

\paragraph{Proof of Theorem~\ref{main}, (ii).}
Note that $X\setminus S$ is Lindel\"of. Hence, we may choose countably many points $p_i$ and take the corresponding $b_i:=b_{p_i}$ and $F_i$ as in Proposition~\ref{isom}, so that $X\setminus S$ is covered by the countable union of $X\setminus b_i^{-1}(0)$. In other words, for
\[
S_k:=\bigcap_{i=1}^k b_i^{-1}(0)
\]
we have $\bigcap_{k=1}^\infty S_k=S$. If $S_N=S$ for some finite $N$, then for $F^{(N)}=F_1\times\cdots\times F_N$ and the Zariski closure $\cA^{(N)}$ of its image, $X\setminus S$ can be written as the union of finitely many Zariski open sets
\[
F^{(N)}(X\setminus b_i^{-1}(0)),\qquad 1\le i\le N,
\]
in $\cA^{(N)}$, and the proof is complete.

To show this, recall the following lemma (\cite{D85}, Lemme~13.2).

\begin{lem}\label{irred}
Let $X$ be an $n$-dimensional complex manifold and $Y$ an analytic set whose maximal dimension is $p$. Then, writing $d=n-p$, we have
\[
H^q(X,X\setminus Y;\mathbb{R})\cong
\begin{cases}
0 & q<2d,\\
\mathbb{R}^J & q=2d,
\end{cases}
\]
where $J$ is an index set in bijection with the $p$-dimensional irreducible components of $Y$.
\end{lem}

Apply this lemma to the complex manifold $X\setminus S$ and its analytic subset $S_k\setminus S$. Let $p_k$ be the maximal dimension of $S_k\setminus S$ and set $d_k=n-p_k$. Consider the relative exact sequence
\[
H^{2d_k-1}(X\setminus S_k)\to H^{2d_k}(X\setminus S,X\setminus S_k)\to H^{2d_k}(X\setminus S).
\]
The first term is finite-dimensional since $X\setminus S_k$ is quasi-affine, and the third term is finite-dimensional by assumption. Hence the middle term is finite-dimensional. By Lemma~\ref{irred}, it follows that $S_k\setminus S$ has only finitely many maximal irreducible components.

Since $\bigcap_{k=1}^\infty S_k=S$, these maximal irreducible components must disappear after finitely many steps. Consequently $p_k$ must decrease after finitely many steps. Repeating the same argument until $S_k\setminus S$ vanishes, we obtain the conclusion.

\paragraph{Proof of Theorem~\ref{main}, (iii).}
We use the same argument as in the final part of Proposition~\ref{complement}. Namely, for the $F^{(N)}$ and $\cA^{(N)}$ constructed above, if a point of $\cA^{(N)}\setminus F^{(N)}(X\setminus S)$ had codimension at least $2$, then since $\cA^{(N)}$ is normal, the Hartogs extension theorem implies that the holomorphic convexity of $F^{(N)}(X\setminus S)$ fails, yielding a contradiction.

\paragraph{Proof of Theorem~\ref{main}, (iv).}
By Proposition~\ref{dense}~(2), $A_\phi^0(X)$ is dense in $\cO(X)$ with respect to the topology of uniform convergence on compact subsets. As in the initial part of the proof of Proposition~\ref{main1}, for any $f \in A_\phi^0(X)$ the function $f \circ (F^{(N)})^{-1}$ is a regular function on $F^{(N)}(X)$ by Lemma~\ref{rationalext} and smoothness. Since $X$ is Stein, this implies that $F^{(N)}(X)$ is rationally convex. Therefore, by Theorem~\ref{Mok}, $F^{(N)}(X)$ is affine.

\part{Algebraization of Entire Grauert Tube}

We now investigate the algebraicity of entire Grauert tubes using Theorem~\ref{main}. This is carried out by constructing explicitly a function $\psi$ satisfying the conditions of Theorem~\ref{main}. The construction relies heavily on the Monge--Amp\`ere foliation of the entire Grauert tube, as well as on the special vector fields along each leaf, which are called \emph{parallel vector fields}.



\section{Review of Grauert tube theory}

In this section we review the basic concepts of Grauert tube theory required for the proof of Theorem \ref{egt}. For more detailed classical references we refer to \cite{LS91}, \cite{S91}, \cite{GS91}, \cite{BK77}, and \cite{BHH03}.
\subsection{Grauert tubes}
The Grauert tube of radius $r$ of a real $n$-dimensional compact Riemannian manifold $(M,g)$ is the complex $n$-dimensional manifold $(T^rM,J)$, where
\[
T^rM:=\{v\in TM:\|v\|_g<r\},
\]
and $J$ is the unique complex structure on $T^rM$, called the \emph{adapted complex structure}. To give the precise definition, we first recall the Riemann foliation on the tangent bundle of a Riemannian manifold.

\begin{defn}\label{rf}
Let $(M,g)$ be a Riemannian manifold. For its tangent bundle $TM$, define the dilation operator
\[
N_\tau:TM\longrightarrow TM,\qquad N_\tau(v)=\tau v.
\]
Given a geodesic $\gamma$, define an immersion $\phi_\gamma:\mathbb{C}\longrightarrow TM$ by
\[
\phi_\gamma(\sigma+i\tau)=N_\tau\bigl(\gamma'(\sigma)\bigr).
\]
The foliation of $TM\setminus M$ whose leaves are $\phi_\gamma(\mathbb{C}\setminus\mathbb{R})$, as $\gamma$ ranges over all geodesics of $M$, is called the \emph{Riemann foliation}.
\end{defn}

Indeed, for every $v\in TM\setminus M$, there exists a geodesic $\gamma$ whose velocity vector is tangent to $v$. Moreover, if two geodesics $\gamma$ and $\gamma'$ have distinct images, then $\phi_\gamma(\mathbb{C}\setminus\mathbb{R})\cap\phi_{\gamma'}(\mathbb{C}\setminus\mathbb{R})=\varnothing$. Thus, the above collection of leaves defines a well-defined foliation.

Recall that $T^rM:=\{v\in TM:\|v\|_g<r\}$. According to Lempert and Sz\H{o}ke \cite{S91,LS91}, for every compact real-analytic Riemannian manifold $(M,g)$, there exists $r>0$ and a unique complex structure $J$ on $T^rM$ such that
\[
\phi_\gamma:\{\sigma+i\tau\in\mathbb{C}:|\tau|<r\}\longrightarrow T^rM
\]
is holomorphic for every geodesic $\gamma$. This leads to the following definition.

\begin{defn}
Let $(M,g)$ be a compact Riemannian manifold.
\begin{enumerate}
    \item The complex structure $J:TT^rM\longrightarrow TT^rM$ for which every leaf map $\phi_\gamma$ of the Riemann foliation is holomorphic is called the \emph{adapted complex structure}.
    \item With respect to the adapted complex structure, each leaf immersion $\phi_\gamma$ into $T^rM$, as well as its image, is called a \emph{complexified leaf} (or simply a \emph{leaf}).
    \item The complex manifold $(T^rM,J)$ endowed with the adapted complex structure is called the \emph{Grauert tube of radius $r$} of $(M,g)$. If the adapted complex structure is defined on the entire tangent bundle $T^\infty M:=TM$, then $(TM,J)$ is called an \emph{entire Grauert tube}.
\end{enumerate}
\end{defn}

The adapted complex structure $(T^r M,J)$ satisfies a special Stein manifold structure, called a \emph{Monge--Amp\`ere model}.

\begin{defn}
Let $X$ be a Stein manifold equipped with a (possibly bounded) strictly plurisubharmonic exhaustion function $u^2$. Assume that $M:=\{u=0\}$ is a smooth submanifold satisfying $\dim_{\mathbb{C}}X=\dim_{\mathbb{R}}M$. If $u$ satisfies the homogeneous complex Monge--Amp\`ere equation (HCMA)
\[
\bigl(dd^c u\bigr)^n=0\quad\text{on}\quad X\setminus M,
\]
then the triple $(X,M,u)$ is called a \emph{Monge--Amp\`ere model}.
\end{defn}

For the adapted complex structure $J$ of a compact Riemannian manifold $(M,g)$, the norm function $u(v)=\|v\|_g$ makes $(T^r M,M,u)$ a Monge--Amp\`ere model \cite[Theorem 5.6]{LS91}. Conversely, suppose that $(X,M,u)$ is a Monge--Amp\`ere model and set $R:=\sup_Xu$. Then there exists a diffeomorphism $\phi:T^RM\to X$ such that the pullback by $\phi$ of the complex structure on $X$ is the adapted complex structure on $T^RM$ associated with the restriction to $M$ of the K\"ahler metric induced by $u^2$ \cite[Theorem 3.1]{LS91}. From this point of view, adapted complex structures and Monge--Amp\`ere models are equivalent objects \cite[Section 5]{LS91}.

Since a Grauert tube $T^rM$ is a subset of the tangent bundle, it carries a natural symplectic structure. Let $\pi:TM\to M$ be the natural projection. Using its differential $d\pi:TTM\to TM$, we define the Liouville $1$-form $\theta$ and the corresponding Poincar\'e $2$-form $\Omega$ at $(p,v)\in TM$ by
\[
\theta_{(p,v)}(-):=g_p(d\pi(-),v), \qquad \Omega:=-d\theta.
\]
Thus $TM$ always carries a canonical symplectic structure. On the domain where the Grauert tube is defined, the complex structure $J$ together with $\Omega$ forms a Kähler structure, and the resulting Kähler metric coincides with the original Riemannian metric $g$ along $M$.



\subsection{Parallel vector fields}
Now we introduce \emph{parallel vector fields} on each complexified leaf, which are key ingredients in the explicit description of the adapted complex structure. Parallel vector fields along leaves are analogous to Jacobi fields on geodesics, and these can be defined in two equivalent ways.

First, we give a formal definition. Viewing $TM\setminus M$ as a smooth manifold, recall that the geodesic flow $\Phi_t$ and the dilation $N_\tau$ are smooth self-maps of $TM\setminus M$. A vector field $\xi:\phi_\gamma\to TTM|_{\phi_\gamma}$ is called a parallel vector field if it is invariant under $d\Phi_t$ and $dN_\tau$, i.e.
\[
d\Phi_t\,\xi(\alpha+i\beta)=\xi(\alpha+t+i\beta),\qquad
dN_\tau\,\xi(\alpha+i\beta)=\xi(\alpha+i\tau\beta).
\]
Note that any point of $\phi_\gamma$ away from the real locus can be moved to any other point of $\phi_\gamma$ by a combination of the geodesic flow and dilation. Hence, the value of a parallel vector field at a single non-real point determines it everywhere.

Parallel vector fields can also be described concretely in analogy with Jacobi fields. Consider a one-parameter family of geodesics $\gamma_s$ with $\gamma_0=\gamma$. This induces a one-parameter family of leaves $\phi_{\gamma_s}$. Then a parallel vector field $\xi:\phi_\gamma\to TTM|_{\phi_\gamma}$ can be defined by
\[
\xi(\sigma+i\tau):=\left.\frac{\partial}{\partial s}\phi_{\gamma_s}(\sigma+i\tau)\right|_{s=0}
=\left.\frac{\partial}{\partial s}N_\tau(\gamma_s'(\sigma))\right|_{s=0}.
\]
Note that when $\tau=0$ this coincides with a Jacobi field. The vector field defined in this way is invariant under geodesic flow and dilation, and conversely every parallel vector field arises in this manner.

The analogy between Jacobi fields on the real part $M$ and parallel vector fields on the Grauert tube $T^rM$ is remarkable. First, the leaves of the Monge--Amp\`ere foliation are totally geodesic and flat in $TM\setminus M$ (\cite{LS91}, Proposition 3.4). Thus, along any geodesic contained in the leaf, any parallel vector field restricts to a Jacobi field. Combining this with the second definition above, we obtain a one-to-one correspondence between normal Jacobi fields along $\gamma$ and parallel vector fields normal to the leaf $\phi_\gamma$. Consequently the space of (normal) parallel vector fields is a real $2n$-dimensional (respectively $2n-2$-dimensional) vector space. Since a concrete basis of this space will be recalled frequently, we record the following setting.

\begin{con}\label{setting}
Assume that a real $n$-dimensional Riemannian manifold $(M,g)$ admits an entire Grauert tube $TM$. Fix $p\in M$ and a unit vector $v\in T_pM$. Let $\gamma$ be the geodesic satisfying $\gamma(0)=p$ and $\gamma'(0)=v$, and let $\phi_\gamma$ be the corresponding leaf. Choose an orthonormal basis $e_1,\dots,e_n$ of $T_pM$ with $e_n=v$.  

Let $\xi_i(\sigma)$ and $\eta_j(\sigma)$ be Jacobi fields with initial conditions
\[
\xi_i(0)=e_i,\qquad \frac{D}{d\sigma}\xi_i(0)=0,
\]
\[
\eta_j(0)=0,\qquad \frac{D}{d\sigma}\eta_j(0)=e_j.
\]
Define the corresponding parallel vector fields $\xi_i(\sigma+i\tau)$ and $\eta_j(\sigma+i\tau)$ along $\phi_\gamma$.  

Let $\Xi:=(\xi_1,\dots,\xi_n)$ and $H:=(\eta_1,\dots,\eta_n)$ be the $n$-tuples of parallel vector fields. Along the geodesic $\gamma$ define the transition matrix $A:\gamma\to M_n(\mathbb{R})$ by $H=\Xi A$.  

Using the complex structure $J$ of the Grauert tube, define the holomorphic parts
\[
\Xi^{1,0}:=(\xi_1^{1,0},\dots,\xi_n^{1,0}),\qquad
H^{1,0}:=(\eta_1^{1,0},\dots,\eta_n^{1,0}),
\]
and the complex transition matrix $\Phi:\phi_\gamma\to M_n(\mathbb{C})$ by
\[
H^{1,0}=\Xi^{1,0}\Phi .
\]
\end{con}

In this setting, the entries of each of $\Xi^{1,0}$ and $H^{1,0}$ form a linearly independent system of holomorphic vector fields on the non-real locus (see \cite{LS91}, proof of Theorem~4.2). Note that $\Phi$ is the analytic continuation of $A$, i.e.\ $\Phi|_{\mathbb{R}}=A$. An immediate consequence of Construction \ref{setting} is that it allows us to specify the complex structure $J$ of the Grauert tube. Since each leaf is a complex manifold and $J$ does not mix the tangent and transverse directions of the leaf (\cite{LS91}, Theorem 5.3), it suffices to determine the action of $J$ on $H$ and $\Xi$. Writing $\Phi=A+iB$, comparison of the real and imaginary parts of $H^{1,0}=\Xi^{1,0}\Phi$ yields
\[
H=\Xi A+J\Xi B, \qquad -JH=\Xi B-J\Xi A,
\]
\begin{equation}\label{J}
J\Xi=(H-\Xi A) B^{-1}, \qquad
JH=H B^{-1} A-\Xi(B+A B^{-1} A).
\end{equation}
The results of \cite{S91} and \cite{LS91} together show that the maximal radius to which $\Phi$ extends while $B:=\Im\Phi$ remains invertible coincides with the maximal radius of the Grauert tube.

Finally we record several properties of the Grauert tube.

\begin{thm}\label{egtprop}
Let $T^rM$ be a Grauert tube with exhaustion $E:=\|v\|_g^2/2$, Liouville $1$-form $\theta$, Poincar\'e $2$-form $\Omega$, and parallel vector fields $\xi_i,\eta_j$ as in Construction \ref{setting}. Then the following hold.
\begin{enumerate}
\item For $(p,v)\in T^rM$, the space
\[
V_{(p,v)}:=\ker\theta_{(p,v)}\cap\ker(dE)_{(p,v)}\subset T_v(TM)
\]
is a real $2n-2$-dimensional $J$-invariant subspace of $T_v(TM)$ (\cite{LS91}, Theorem 5.3).
\item $\bar\partial E-\partial E=i\theta$, and $\idd E=\frac12\Omega$ (\cite{LS91}, Corollary 5.5).
\item $\Omega(\xi_i,\xi_j)=\Omega(\eta_i,\eta_j)=0$ (\cite{LS91}, Proposition 6.6).
\item On $\sigma+i\tau=i$,
\[
\Omega(\xi_i,\eta_j)=
\begin{cases}
1 & i=j,\\
0 & i\neq j,
\end{cases}
\]
(\cite{BL18}, proof of Proposition 2).
\end{enumerate}
\end{thm}

\section{Construction of complexified metric $\gC$}\label{kcon}
As a key ingredient in constructing the Ricci curvature bound $\psi$, we prove that the metric $g$ on the base manifold $(M,g)$ admits a meromorphic nondegenerate $(2,0)$-extension to the Grauert tube $T^r M$.
\begin{dfnthm}\label{LS}
Let $(M,g)$ be a Riemannian manifold and $T^r M:=\{v \in TM~|~\|v\|_g<r\}$ be a Grauert tube. There exists the nondegenerate meromorphic $(2,0)$-tensor $\gC$ on $T^rM$ obtained as the extension of $g$ to the Grauert tube. We call $\gC$ the \emph{complexified metric}.
\end{dfnthm}
The construction of $\gC$ proceeds as follows. First, since any metric admitting an entire Grauert tube is real-analytic, the metric $g$ extends at least to a small neighborhood $T^\varepsilon M$ of $M$ (Proposition~\ref{epsilon}). Second, on each leaf, parallel vector fields form a holomorphic basis, which allows one to extend $g$ meromorphically along each leaf (Proposition~\ref{leafwise}). Third, as the leaves vary, the extension $\gC$ depends real-analytically on the leaf parameter, and hence is globally real-analytic away from the pole set; since it is holomorphic on $T^\varepsilon M$, it follows that $\gC$ is holomorphic (Proposition~\ref{global}).

Before giving the precise proof, we recall the following key facts.
\begin{rem}\label{continue}
Recall the following two basic facts.
\begin{enumerate}[label=(\arabic*), leftmargin=2.2em]
\item Let $D$ be a connected complex domain. If two real-analytic functions on $D$ agree on a nonempty open subset $U \subset D$,
then they agree on all of $D$.
\item Let $D$ be a connected complex domain and let $S \subset D$ be a totally real subset of maximal dimension.
If two holomorphic functions on $D$ agree on $S$, then they agree on all of $D$.
\end{enumerate}
\end{rem}
\begin{prop}\label{epsilon}
Let $g$ be the metric tensor on $M$. Then there exists $\varepsilon>0$ such that on the Grauert tube $T^\varepsilon M$ there exists a holomorphic $(2,0)$-tensor $\gC$ extending $g$.
\end{prop}

\begin{proof}
Using that $M$ is a real-analytic totally real submanifold of $TM$, for each point of $M$ choose a local holomorphic coordinate system $\{z_i\}_{i=1}^n$ on $TM$ such that $x_i:=\Re(z_i)$ restricts to a real-analytic coordinate system on $M$.
In these coordinates we may write
\[
g=\sum_{i,j=1}^n g_{ij}(x)\,dx_i\otimes dx_j.
\]
Since $M$ is real-analytic, after shrinking the chart if necessary each coefficient $g_{ij}$ admits a holomorphic extension to a neighborhood of $M$ in $TM$. Denote these extensions by $g_{\bC ij}(z)$ and define
\[
\gC:=\sum_{i,j=1}^n g_{\bC ij}(z)\,dz_i\otimes dz_j.
\]
Then $\gC|_M=g$ because $dz_i|_M=dx_i$.

Next choose $\varepsilon>0$ sufficiently small so that $T^\varepsilon M$ is covered by such neighborhoods. It remains to check that $\gC$ is globally well-defined on $T^\varepsilon M$. On the overlap of two charts $C_1$ and $C_2$, the locally defined tensors restrict to the same $g$ on $M$. Hence the difference $\gC^{(1)}-\gC^{(2)}$ vanishes on $M\cap(C_1\cap C_2)$. By Remark \ref{continue}(2) applied componentwise, it vanishes on all of $C_1\cap C_2$. Therefore the transition is consistent and $\gC$ is globally well-defined on $T^\varepsilon M$.
\end{proof}

\begin{prop}\label{leafwise}
Let $\gamma$ be a geodesic in $M$ and let $\phi_\gamma$ denote its complexified geodesic in $T^rM$. Then $\gC$ extends to all of $\phi_\gamma$ as a meromorphic nondegenerate $(2,0)$-tensor.
\end{prop}

\begin{proof}
We first consider the geometry on the real locus $(M,g)$. For $n$-tuples of vector fields $V=(V_1,\dots,V_n)$ and $U=(U_1,\dots,U_n)$, define the matrix $g(V,U)$ as
\[
g(V,U)_{ij} := g(V_i, U_j).
\]
If $V$ and $U$ consist of Jacobi fields, then their Wronskian
\[
W(V,U) := g(V',U) - g(V,U')
\]
is constant. Indeed, writing $R$ for the Jacobi operator, we compute
\[
\begin{aligned}
W'(V,U) &= \bigl(g(V',U) - g(V,U')\bigr)' = g(V'',U) + g(V',U') - g(V',U') - g(V,U'') \\
&= g(RV,U) - g(V,RU) = 0.
\end{aligned}
\]

Now consider Construction~\ref{setting}. On the real locus at $\sigma=0$, we note that $W(\Xi,\Xi)=0$ and $W(H,\Xi)=I$. Since $H = \Xi A$, we have
\[
\begin{aligned}
W(H,\Xi)
&= g(H',\Xi) - g(H,\Xi') = g(\Xi' A + \Xi A',\Xi) - g( \Xi A,\Xi') \\
&= g(\Xi,\Xi) A' - W(\Xi,\Xi)A = g(\Xi,\Xi) A' = I.
\end{aligned}
\]
Thus, on the real locus (where defined), we obtain
\[
g(\Xi,\Xi) = (A')^{-1}.
\]

Recall that the entries of $\Xi^{1,0}$ form a holomorphic frame along each leaf $\phi_\gamma$ away from the real locus. For $\Phi$, the holomorphic extension of $A$, define
\[
\gC(\Xi^{1,0}, \Xi^{1,0}) := (\Phi')^{-1},
\]
on the non-real locus of $\phi_\gamma$, and set $\gC|_{\mathbb{R}} = g$. Since $\Phi'$ is well-defined on the non-real locus, this defines $\gC$ as a meromorphic nondegenerate $(2,0)$-tensor.

Finally, this extension does not depend on the choice of Construction \ref{setting}. For any such choice, on $T^\varepsilon M\setminus M$, the resulting extension of $g$ agrees with the tensor $\gC$ already defined on $T^\varepsilon M$ from Proposition \ref{epsilon}.
By Remark \ref{continue} (1), any two extensions obtained from different choices coincide on the whole leaf.
\end{proof}

\begin{rem}\label{n-1}
When $\gC$ is expressed as a matrix with respect to the parallel vector fields $\Xi$, the essential information is contained in the $(n-1)\times(n-1)$ block. On the real locus, note that $g(\xi_n,\xi_n)=1$ and $g(\xi_i,\xi_n)=0$ for $1 \le i \le n-1$, since $\xi_i$ is a normal Jacobi field. Hence, for its extension $\gC$, one has $\gC(\xi_n^{1,0}, \xi_n^{1,0})=1$ and $\gC(\xi_i^{1,0}, \xi_n^{1,0})=0$ for $1 \le i \le n-1$. 

Accordingly, we will denote the $(n-1)\times(n-1)$ block of $\Phi'$ corresponding to the normal direction by $\Phi'_\perp$.
\end{rem}

\begin{prop}\label{global}
The tensor $g$ extends to all of $T^r M$ as a meromorphic nondegenerate $(2,0)$-tensor $\gC$.
\end{prop}
\begin{proof}
Since $T^r M\setminus M$ is uniquely covered by the leaf foliation given by complexified geodesics $\phi_\gamma$, we define $\gC$ on each leaf by the extension constructed in Proposition \ref{leafwise} and verify that the resulting tensor is meromorphic on $T^r M$.
Fix a unit-speed geodesic $\gamma$ with $\gamma(0)=p$, $\gamma'(0)=v$, and an orthonormal frame $v_1,\dots,v_{n-1}$ perpendicular to $v$.
Geodesics near $\gamma$ can be described by the following $(2n-2)$-parameter family.
Let $N_vM$ be the normal space to $v$ in $T_pM$ and take $u,w\in N_vM$.
Let $P_u$ denote parallel transport along the curve $t\mapsto \exp_p(tu)$ from $p$ to $\exp_p u$.
Define
\[
\gamma(u,w)(s):=\exp_{\exp_pu}\Bigl(s\,P_u\!\left(\frac{v+w}{\sqrt{1+\|w\|^2}}\right)\Bigr).
\]
Using this parametrization we can choose the nearby geodesic $\gamma(u,w)$ and its initial point $\gamma(u,w)(0)$ real-analytically in $(u,w)$.
Next select an orthonormal frame $v_i(u,w)$ perpendicular to $\gamma(u,w)'(0)=v(u,w)$ as follows.
Apply the Gram--Schmidt process to $(v+w,v_1,\dots,v_{n-1})$ to obtain
\[
\frac{v+w}{\|v+w\|},\;v_1',\dots,v_{n-1}'.
\]
Define
\[
v_i(u,w):=P_u(v_i').
\]
This choice depends real-analytically on $(u,w)$, hence the corresponding parallel vector fields along $\gamma(u,w)$ vary analytically.
Therefore $\gC$ varies analytically and thus is real-analytic on $T^r M$.
On the tube $T^\varepsilon M$, Proposition \ref{epsilon} gives that $\gC$ is holomorphic, hence $\bar\partial \gC=0$ on $T^\varepsilon M$.
Since $\bar\partial \gC$ is real-analytic on $T^r M$, by Remark \ref{continue} (1) we conclude $\bar\partial \gC=0$ where $\gC$ is finite.
Thus $\gC$ is meromorphic on $T^r M$, and by Proposition \ref{leafwise}, it is nondegenerate away from its pole set.
\end{proof}






\section{Algebraization of entire Grauert tube}
We now construct a strictly plurisubharmonic exhaustion $\phi$ and a Ricci curvature bound $\psi$ for Theorem \ref{main}.
\subsection{Strictly plurisubharmonic exhaustion $\phi$}
\begin{defn}
Let $(M,g)$ be a compact Riemannian manifold admitting an entire Grauert tube $TM$. 
For $E(v):=\tfrac{1}{2}\|v\|_g^2$ and $u(v):=\sqrt{2E(v)}$ on $TM$, define
\[
\phi := \log\bigl(1+\cosh(u)\bigr) = \log\bigl(1+\cosh(\sqrt{2E})\bigr).
\]
\end{defn}

\begin{prop}
Let $\mathrm{Vol}(B_1^n)$ denote the Euclidean volume of the unit ball in $\mathbb{R}^n$. Then
\[
\int_{TM} (dd^c\phi)^n = n!\,\mathrm{Vol}(M,g)\,\mathrm{Vol}(B_1^n).
\]
\end{prop}

\begin{proof}
A direct computation gives
\[
dd^c\phi=2\idd\phi=\frac{\sinh u}{u(1+\cosh u)}\,2\idd E+
\frac{u-\sinh u}{u^3(1+\cosh u)}
2i\,\partial E\wedge\bar\partial E .
\]
Using the Monge--Amp\`ere equation $(\idd u)^n=(\idd\sqrt{2E})^n=0$, we have
\begin{align}
(\idd u)^n 
&= (\idd \sqrt{2E})^n = \Bigl(\frac{1}{\sqrt{2E}}\idd E - \frac{1}{2E\sqrt{2E}}\, i\partial E \wedge \bar\partial E\Bigr)^n\notag \\
&= \frac{1}{u^n}(\idd E)^n - \frac{n}{u^{n+2}}(\idd E)^{n-1}\wedge (i\partial E \wedge \bar\partial E) = 0. \label{MA}
\end{align}
In addition, $2\idd E=\Omega$ (Theorem \ref{egtprop}, (2)), hence we have
\[
(dd^c\phi)^n=\frac{\sinh^{\,n-1}u}{u^{\,n-1}(1+\cosh u)^n}(2\idd E)^n=\frac{\sinh^{\,n-1}u}{u^{\,n-1}(1+\cosh u)^n} \Omega^n .
\]

The symplectic volume form $\frac{\Omega^n}{n!}$ decomposes at $(x,v)\in TM$ as
\[
\frac{\Omega^n}{n!}=d\mathrm{Vol}_g(x)\,dV_x(v),
\]
where $dV_x$ denotes the Euclidean measure on $T_xM$. Writing $dV_x$ in radial form and noting that $u$ corresponds to the fiberwise norm, we obtain
\[
dV_x=u^{n-1}\,du\,dS_x,
\]
where $dS_x$ is the standard area measure on the unit sphere $S_x^{n-1}$.

Combining these expressions yields
\begin{align*}
\int_{TM}(dd^c\phi)^n
&=\int \frac{\sinh^{\,n-1}u}{(1+\cosh u)^n}\,n!\,du\,dS_x\,d\mathrm{Vol}_g\\
&=n!\,\mathrm{Vol}(S^{n-1}_1)\,\mathrm{Vol}(M,g)\int_0^\infty \frac{\sinh^{\,n-1}u}{(1+\cosh u)^n}du\\
&=n!\,\mathrm{Vol}(S^{n-1}_1)\,\mathrm{Vol}(M,g)\left [\frac{1}{n}\tanh^n\frac{u}{2} \right ]_{u=0}^\infty=n!\,\mathrm{Vol}(B_1^n)\,\mathrm{Vol}(M,g).
\end{align*}
\end{proof}

\subsection{Ricci curvature bound $\psi$}

Finally, we complete the proof of Theorem~\ref{egt} by constructing the following function $\psi$ and applying Theorem~\ref{main}\,(i).

\begin{defn}
Let $\beta := dd^c e^\phi = dd^c(1+\cosh(u)) = dd^c(1+\cosh(\sqrt{2E}))$. Let $\gC$ be the complexified metric. For the meromorphic nonvanishing section $\det(\gC)$ of $K_{TM}^{2}$, we define the Ricci curvature bound function $\psi$ by
\[
\psi := -\log \| \det(\gC) \|_\beta.
\]
\end{defn}

Locally, $\psi$ can be written as $\psi = \log |\det(\beta)| - \log |\det(\gC)|$. Then $\idd \log |\det(\beta)| = -\Ric(\beta)$, while the second term is pluriharmonic away from the poles of $\det(\gC)$, i.e.\ the zero locus of $\det(\Phi')$, and defines a plurisubharmonic function in the sense of $(1,1)$-currents globally. Therefore, $\idd \psi + \Ric(\beta) \ge 0$. Hence, in order to apply Theorem~\ref{main}~(i), it remains to estimate the growth of $\psi$.

\begin{prop}
There exist constants $A,B>0$ such that $\psi \le A \phi + B$.
\end{prop}

\begin{proof}
Assume Construction~\ref{setting}. It suffices to show that there exist constants $A,B>0$ independent of the choice of leaf $\phi_\gamma(\sigma+i\tau)$. By conjugation symmetry, it suffices to work on the upper half-plane $\tau>0$, where $u=\tau$.

Locally, $\psi = \log |\det(\beta)| - \log |\det(\gC)|$. On each leaf, away from the real locus (and additionally near $\sigma+i\tau=0$), the vector fields $\xi_1^{1,0},\dots,\xi_n^{1,0}$ form a holomorphic frame. We compute each term with respect to this basis.

\medskip
\noindent
(1) The term $-\log |\det(\gC)|$.
By Proposition \ref{leafwise}, on the basis $\xi_1^{1,0},\dots,\xi_n^{1,0}$, we have $\det(\gC)=\det((\Phi')^{-1})=1/\det(\Phi')$.

\medskip
\noindent
(2) The term $\log |\det(\beta)|$.
Writing $\beta$ in terms of $E$ and $u=\sqrt{2E}$, we have
\begin{align*}
\beta 
&= dd^c e^\phi = 2\idd(1+\cosh\sqrt{2E}) \\
&= \frac{1}{u}\sinh(u)\,2\idd E 
+ \Bigl(\frac{1}{u^2}\cosh(u) - \frac{1}{u^3}\sinh(u)\Bigr)\, 2i\partial E \wedge \bar\partial E.
\end{align*}

Combining equation \ref{MA} with $2\idd E = \Omega$ (Theorem~\ref{egtprop}~(2)), a direct computation yields
\[
\beta^n = \frac{\cosh u \, \sinh^{n-1} u}{u^{\,n-1}} \, \Omega^n.
\]

Next, we compute $\Omega(\xi_i^{1,0},\xi_j^{0,1})$. Since $\idd E = \tfrac{1}{2}\Omega$, the form $\Omega$ has the same symmetry as a Kähler form. Together with $\Omega(\xi_i,\xi_j)=0$ (Theorem~\ref{egtprop}~(3)), we obtain
\[
\Omega(\xi_i^{1,0},\xi_j^{0,1}) = \frac{i}{2}\Omega(\xi_i, J\xi_j).
\]

Recall that at $\phi_\gamma(i)$, one has $\Omega(\xi_i,\eta_i)=1$ and $\Omega(\xi_i,\eta_j)=0$ for $i\ne j$ (Theorem~\ref{egtprop}~(4)). Since the geodesic flow is Hamiltonian, $\Omega$ is invariant under the flow, and by definition, $\Omega$ scales by a factor $\tau$ under the dilation $N_\tau$. Parallel vector fields are also invariant under both operations. Hence, at $\phi_\gamma(\sigma+i\tau)$ we obtain $\Omega(\xi_i,\eta_i)=\tau$ and $\Omega(\xi_i,\eta_j)=0$ for $i\ne j$.

Note that on $\phi_\gamma(\sigma+i \tau)$, $u=\tau$. Using $J\Xi = (H - \Xi A)B^{-1}$ (Equation~\ref{J}) and $\Omega(\xi_i,\xi_j)=0$, we conclude that on the frame $\{\xi_i^{1,0}\}$,
\[
\det(\Omega) = 2^{-n} u^n \det(B^{-1}), \qquad
\det(\beta) = 2^{-n}u \cosh u \, \sinh^{n-1} u \, \det(B^{-1}).
\]

Therefore, on $\phi_\gamma(\sigma+i\tau)$, the function $\psi$ can be expressed as
\begin{equation}
    \psi = \log\Bigl(2^{-n} u \cosh u \, \sinh^{n-1} u \left | \frac{\det(\Phi')}{\det(B)}\right | \Bigr) \label{psiasymp}.
\end{equation}
Finally, the upper bound of 
$\bigl|\det(\Phi')/\det(B)\bigr|
= \bigl|\det(\Phi')/\det(\Im \Phi)\bigr|$
is obtained from the following matrix-valued Schwarz--Pick lemma.

\begin{lem}
Let $\mathbb{C}^+ := \{z \in \mathbb{C} \mid \Im z > 0\}$, and let 
$\Phi : \mathbb{C}^+ \to M_n(\mathbb{C})$
be a symmetric, matrix-valued Herglotz function, i.e., a matrix-valued holomorphic function such that $B := (\Phi - \Phi^*)/(2i)$ is positive definite. Then for $z = \sigma + i\tau$,
\[
\left|\frac{\det(\Phi'(z))}{\det(B(z))}\right| \le \frac{1}{\tau^n}.
\]
\end{lem}

\begin{proof}
By the matrix-valued Herglotz representation \cite{GT00},
\[\Phi(z) = C + Dz + \int_{\mathbb{R}} \bigl( \frac{1}{t - z} - \frac{t}{1+t^2} \bigr)\, d\mu(t),\]
where $C,D$ are Hermitian matrices, $D \ge 0$, and $\mu$ is a positive semidefinite matrix-valued Borel measure. Since $\Phi$ is symmetric, they are real symmetric. Hence, we have the following.
\[\Phi'(z) = D + \int_{\mathbb{R}} \frac{1}{(t - z)^2}\, d\mu(t) \qquad B(z) = \tau D + \int_{\mathbb{R}} \frac{\tau}{|t - z|^2}\, d\mu(t).\]

For any $v \in \mathbb{C}^n$,
\[
|v^* \Phi'(z) v|
\le |v^* D v| + \int_{\mathbb{R}} \frac{1}{|t - z|^2}\, d(v^* \mu(t) v)
= \frac{1}{\tau} \, v^* B(z) v.
\]

Setting $v = B^{-1/2} w$ gives
$|w^* B^{-1/2} \Phi'(z) B^{-1/2} w| \le \frac{1}{\tau} \|w\|^2$,
so the absolute values of all eigenvalues of $B^{-1/2} \Phi'(z) B^{-1/2}$ are bounded by $1/\tau$. 
Therefore,
\[
\left|\det (B^{-1/2} \Phi'(z) B^{-1/2})\right |=\left|\frac{\det(\Phi'(z))}{\det(B(z))}\right| \le \frac{1}{\tau^n}.
\]
\end{proof}
On the real locus of Construction~\ref{setting}, the Jacobi field systems $\Xi$ and $H$ have vanishing Wronskian. Hence the transition matrix $A$ relating them is symmetric. Consequently, its extension $\Phi$ is also symmetric, and by \cite[Lemma~6.7]{LS91}, $B=\Im \Phi$ is positive definite. Therefore the above lemma applies. 

Since both the bound $1/\tau^n = 1/u^n$ and Equation~\ref{psiasymp} are independent of the choice of leaf $\phi_\gamma$, it follows that $\psi$ is bounded above linearly in $\phi\sim u$.
\end{proof}

\section{Tube singularity $S$ and Theorem \ref{egt}}\label{egtdisc}

We conclude this paper with a discussion of Theorem \ref{egt} and possible directions for future improvement. We begin by introducing the following definition.

\begin{defn}\label{ts}
Let $(M,g)$ be a Riemannian manifold, and let $T^rM$ be its Grauert tube equipped with the complexified metric $\gC$. The pole set
\[
S \subset T^rM \setminus M
\]
of $\det(\gC)$ is called the \emph{tube singularity}.
\end{defn}

The tube singularity $S$ is a natural analytic subset of a Grauert tube: it depends only on the underlying metric $g$. In addition, for $\psi := -\log \|\det(\gC)\|_{\beta}$, note that the tube singularity $S$ coincides with the set
\[
S := \{x \mid e^{-2\psi} \notin L^1_\loc(x)\}
\]
appearing in Theorem \ref{main}. Indeed, since $e^{-2\psi} = \|\det(\gC)\|_{\beta}^2$ and $\det(\gC)$ is locally represented by a meromorphic function, its pole set is precisely the set where $\|\det(\gC)\|_{\beta}^2$ fails to be locally integrable.

Theorem \ref{egt} shows that an entire Grauert tube $TM$ is affine outside a codimension-one subset of the form $b_p^{-1}(0)$. Moreover, $TM\setminus S$ is covered by the affine open sets $TM\setminus b_p^{-1}(0)$, which are all birational to one another. Since the tube singularity $S$ is locally given as the zero locus of a holomorphic function at each point, namely by $(\det \gC)^{-1} = 0$, the complement $TM \setminus S$ is Stein \cite[Corollary 4]{Z21}. Therefore, Theorem \ref{egt} can be roughly rephrased as follows.

\medskip
\noindent\textbf{Theorem~\ref{egt} (rephrased)}
\textit{For an entire Grauert tube $TM$ with tube singularity $S$, the complement $TM\setminus S$ is a Stein manifold covered by birationally equivalent affine Zariski open subsets.}\medskip

However, each hypersurface $b_p^{-1}(0)$ necessarily contains the entire tube singularity $S$. In addition, Theorem \ref{egt} is obtained from the first step (i) of Theorem \ref{main}; hence, to apply the additional statements of that theorem, one needs sufficient control over $S$. Therefore, a complete resolution of the Burns conjecture depends on understanding the structure of the tube singularity $S$.

Indeed, if $S$ is empty, then by Theorem \ref{main} (iv), the entire Grauert tube is affine. However, there exist explicit examples where $S$ is nonempty, which constitutes the main obstruction to the Burns conjecture. Moreover, if the $(n,0)$-part $V^{n,0}$ of the analytic continuation of the volume form $V$ of the real locus $(M,g)$ exists, then one must have
\[
\|V^{n,0} \wedge \overline{V^{n,0}}\|_\beta = C_n\|\det(\gC)\|_\beta,
\]
for a positive dimension-dependent constant $C_n$. Therefore, the presence of $S$ also reflects the difficulty encountered in the approach of Aguilar and Burns \cite{AB00}. We first describe such an example and then suggest a known control on $S$.

Sz\H oke \cite[Section 4]{S91} constructed a family of surfaces of revolution $M_d$ that all admit entire Grauert tubes, and proved that, up to rescaling, these exhaust all such surfaces of revolution.

\begin{exmp}
For $d\geq 0$, let $C_d = (\chi_1,0,\chi_2) \subset \bR^3$ be a curve defined by
\[
\chi_1(s) := \frac{\sin s}{\sqrt{d \sin^2 s + 1}}, \qquad (\chi_1')^2 + (\chi_2')^2 = 1, \qquad -\chi_2(0) = \chi_2(\pi).
\]
Let $M_d$ be the surface of revolution obtained by rotating $C_d$ around the third coordinate axis.
\end{exmp}

As observed in Proposition \ref{leafwise}, on each complexified leaf $\phi_\gamma$, the tube singularity $S$ corresponds to the zero set of $\det(\Phi')$, where $\Phi$ is the transition matrix between parallel vector fields. The above example provides an explicit case where $\Phi'$ degenerates outside the real locus when $d \neq 0$.

Recall that on a surface of revolution, meridian curves are geodesics. Sz\H oke showed the following fact \cite[Proposition 4.2]{S91}.

\begin{prop}
Let $\xi$ be the unit normal vector field along the meridian curve $C_d$ in $M_d$. Then
\[
X_1(s) := \frac{\sin s}{\sqrt{d \sin^2 s + 1}} \, \xi(s), \qquad
X_2(s) := \frac{ds \sin s - \cos s}{\sqrt{d \sin^2 s + 1}} \, \xi(s)
\]
are Jacobi fields.
\end{prop}

Interpreting this in the framework of Construction \ref{setting} and Remark \ref{n-1}, note that since the base manifold is two-dimensional, the derivative of the transition matrix between two normal Jacobi field systems $\Phi'_\perp$ reduces to the ratio of two Jacobi fields. In this case, $X_1$ corresponds to $\eta_1$ and $X_2$ to $\xi_1$, so that the $1 \times 1$ matrix $\Phi'_\perp$ is given by
\[
\Phi'_\perp(z) = \left( \frac{\sin z}{dz \sin z - \cos z} \right)'.
\]
A direct computation shows that $\Phi'_\perp$ vanishes at
\[
z = n\pi \pm i \, \mathrm{arcsinh}\!\left(\frac{1}{\sqrt{d}}\right), \qquad n \in \bZ.
\]
Thus, except for the round sphere $M_0$, every surface in the family $M_d$ has a nonempty tube singularity $S$.

Lempert \cite[Proposition 3.1, Lemma 3.1, Lemma 4.1]{L93} showed that the absence of zeros of $\Phi'$ near $M$ can be partially controlled by the geometry of $(M,g)$, providing an estimate on the region where $S$ does not appear.

\begin{prop}
Let $\rho$ be the conjugate radius of $(M,g)$, i.e., the minimal distance between distinct conjugate points along any geodesic. Then
\[
T^{\rho/2}M \cap S = \emptyset.
\]
\end{prop}

However, to the best of the author's knowledge, there is currently no known method to control the topology of $S$ sufficiently to apply the global argument in Theorem \ref{main}. Therefore, in view of the Burns conjecture and its applications to the good complexification conjecture, we propose the following questions.

\begin{ques}
Let $TM$ be an entire Grauert tube with tube singularity $S$. Can one guarantee that $\dim H^{2k}(TM \setminus S; \bR) < \infty$ for all $k \in \bZ$?
\end{ques}
As mentioned above, $TM\setminus S$ is a Stein manifold. Therefore, if the above question admits an affirmative answer, then Theorem~\ref{main}(iii) would imply the following global affineness statement for entire Grauert tubes: there exists an irreducible normal affine variety $X$ and an algebraic hypersurface $H \subset X$ such that $TM \setminus S$ is biholomorphic to $X \setminus H$.
\begin{ques}
Let $(M,g)$ be a compact Riemannian manifold admitting an entire Grauert tube. Does there exist a deformation of the metric, preserving the existence of the entire Grauert tube, such that the resulting tube singularity $S$ vanishes?
\end{ques}
Note that the entire Grauert tube $TM$ has the same cohomology as the compact Riemannian manifold $(M,g)$. Therefore, if the tube singularity set $S$ vanishes, then $TM$ is affine by Theorem~\ref{main}(iv). Consequently, if the above question admits an affirmative answer, we obtain that any compact manifold $(M,g)$ admitting an entire Grauert tube has a good complexification.
\bibliography{KSong}
\bibliographystyle{plain}

\vskip 3mm
\noindent
Department of Mathematics, Rutgers University, Piscataway, NJ 08854-8019.

\noindent
{\it Email:} ks1951@rutgers.edu

\end{document}